\documentclass[reqno]{amsart}

\usepackage[top=50truemm,bottom=35truemm,left=30truemm,right=25truemm]{geometry}
\allowdisplaybreaks[4]
\usepackage{amsmath, amssymb}
\usepackage{amsthm}
\usepackage{enumitem}
\usepackage{mathrsfs}

\usepackage{hyperref}
\hypersetup{colorlinks=true,urlcolor=blue,linkcolor=blue,citecolor=blue}

\usepackage{url}

\setlist[enumerate]{label={\upshape (\roman*)}}

\makeatletter
	\@addtoreset{equation}{section}

\theoremstyle{plain}
\newtheorem{theorem}{\indent\sc Theorem}[section]
\newtheorem{lemma}[theorem]{\indent\sc Lemma}
\newtheorem{corollary}[theorem]{\indent\sc Corollary}
\newtheorem{proposition}[theorem]{\indent\sc Proposition}

\theoremstyle{definition}
\newtheorem{definition}[theorem]{\indent\sc Definition}
\newtheorem{remark}[theorem]{\indent\sc Remark}
\newtheorem{example}[theorem]{\indent\sc Example}

\newtheorem{assumption}[theorem]{\indent\sc Assumption}

\begin{document}
\title[Homeomorphism of the Revuz correspondence]{Homeomorphism of the Revuz correspondence for finite energy integrals}
\author[T. Ooi]{Takumu Ooi}\thanks{Department of Mathematics, Faculty of Science and Technology, Tokyo University of Science, 2641 Yamazaki, Noda-shi, Chiba prefecture, 278-8510, Japan.\ E-mail:ooitaku@rs.tus.ac.jp}
\begin{abstract}
We provide necessary and sufficient conditions for the convergence of Revuz measures of finite energy integrals. More precisely, the Revuz map from the set of all smooth measures of finite energy integrals, equipped with the topology induced by the norm given by the sum of the Dirichlet form and the \(L^2(m)\)-norm, to the space of positive continuous additive functionals, equipped with the topology induced by the \(L^2(\mathbb{P}_{m+\kappa+\nu_0})\)-norm with the local uniform topology, is a homeomorphism, where \(m\) is the underlying measure, \(\kappa\) is the killing measure of a Dirichlet form and \(\nu_0\) is an energy functional for the part that the process continuously escaping to the cemetery point.
\end{abstract}

\maketitle    

\section{Introduction}\label{intro}
In the analysis of Markov processes, the properties of positive continuous additive functionals (PCAFs in abbreviation), which are non-negative one-dimensional continuous processes, have been extensively studied. In particular, PCAFs are used to construct and analyze new Markov processes that are time-changed by the original processes. For example, the Cauchy process (1-stable process) on the circle is realized as a time-changed process of the reflected Brownian motion on the disk by the local times on the boundary (\cite[pp.194-196]{CF}). 
Another example is provided by the Feynman–Kac formula, which can be interpreted as the semigroup of a process killed at an exponentially distributed time, where the killing rate is given by a PCAF (\cite[\S A.3.2]{CF}). In Dirichlet form theory, PCAFs correspond one-to-one to smooth measures, Borel measures charging no set of zero capacity for a process. A smooth measure is also called the Revuz measure. This one-to-one correspondence between PCAFs and smooth measures is called the Revuz correspondence, and the bijective map from the space of Revuz measures to the space of PCAFs is called the Revuz map. The origins of this correspondence can be traced back to classical potential theory: Mayer \cite{M}, McKean and Tanaka \cite{MT}, Volkonski\u{\i} \cite{V}, Wentzell \cite{We}, and Blumenthal and Getoor \cite{BG2} studied this correspondence in the setting of Brownian motion and other processes. Later, Azéma, Kaplan-Duflo and Revuz \cite{AKR} developed the correspondence for recurrent Hunt processes, and Revuz \cite{R} formulated in the context of Dirichlet form theory. Furthermore, the correspondence between measures and (possibly discontinuous) additive functionals has been studied in greater generality for (non-symmetric) Markov processes by Azéma \cite{A}, Dynkin \cite{D}, Sharpe \cite{S}, Beznea and Boboc \cite{BB1, BB2, BB3, BBb}, Fitzsimmons and Getoor \cite{FG2, FG}, and others.

This raises the question of whether these two objects correspond at a topological level, specifically, whether the Revuz map can be shown to be a homeomorphism when suitable topologies are imposed on both spaces. The analysis of the convergence of PCAFs plays a crucial role in the construction and convergence of time-changed processes. For instance, verifying the convergence of PCAFs is an essential step in constructing Liouville Brownian motion (\cite{AK, GRV}) and in proving convergence to Liouville Brownian motion (\cite{O}, for example). In these cases, the limiting smooth measures are singular with respect to the underlying measures, and so the corresponding PCAFs cannot be expressed in explicit form. This makes it difficult to construct and prove convergence of these time-changed processes by considering convergence of PCAFs. On the other hand, in the framework of resistance forms, Croydon, Hambly, and Kumagai \cite{CHK}, Croydon \cite{Cr}, and Noda \cite{N1} developed a general theory of the convergence of time-changed processes by considering convergence in distribution of local times, which are PCAFs associated with Dirac measures.

In considering the continuity of the Revuz map, it is necessary to introduce an appropriate topology on the space of smooth measures. In classical potential theory, such as the Newton potentials (associated with Brownian motion) and the Riesz potentials (associated with \(\alpha\)-stable processes), topologies on classes of smooth measures have been constructed and studied (see \cite{L}). In the context of Dirichlet form theory, Nishimori, Tomisaki, Tsuchida, and Uemura \cite{NTTU} generalized these classical constructions and introduced a topology on a class of smooth measures of finite energy integrals. This topology is induced by the sum of the norm of a Dirichlet form and the \(L^2\)-norm with respect to the underlying measure of the Dirichlet form.
They also proved this topological space is a complete separable metric space, and showed the compactness of the Revuz map from the set of all smooth measures of finite energy integrals with this topology, to the space of PCAFs equipped with almost sure convergence for any quasi-every starting point, with the local uniform topology. As another approach, Noda \cite{N} considered the continuity of the Revuz map on a class of smooth measures that converge uniformly under the assumption called the absolutely continuous condition. In a related topic, we also mention that BenAmour and Kuwae \cite{BK} established the convergence of a sequence of semigroups and resolvents of time-changed processes by PCAFs belonging to the Kato class, under the assumption of a certain type of convergence of Green's functions.

 In this paper, we fix a regular Dirichlet form, and we consider the bicontinuity of the Revuz map restricted to the space \(\mathcal{S}_0\) of all smooth measures of finite energy integrals. As a first main result, we prove the topology \(\rho\) on \(\mathcal{S}_0\) introduced by \cite{NTTU} corresponds to the topology induced by \(L^2(\mathbb{P}_{m+\kappa+\nu_0})\) with the local uniform topology, and under these topologies, the Revuz map is bicontinuous, and hence a homeomorphism.
\begin{theorem}\label{main2}
For \(\mu _n, \mu \in \mathcal{S}_0\), let \(A^n, A\) be the corresponding PCAFs respectively. Revuz measures \(\mu_n\) converge to \(\mu\) in \(\rho\) if and only if \(A^n\) converge to \(A\) in \(L^2(\mathbb{P}_{m+\kappa+\nu_0})\) with the local uniform topology, that is, for any \(T>0,\)
\[\lim_{n\to \infty}\mathbb{E}_{m+\kappa+\nu_0}\left[\sup_{0\leq t\leq T}|A^n_t-A_t|^2\right]=0.\]
\end{theorem}

This theorem concerning the homeomorphism of the Revuz map is proved in Section \ref{mainsec} by capturing the following energy of PCAFs.
\begin{proposition}\label{main2-2}
For \(\mu, \nu \in \mathcal{S}_0\), let \(A, B\) be the corresponding PCAFs respectively. It holds that 
\begin{equation}
\mathbb{E}_{\alpha m+\frac{\kappa}{2}+\frac{\nu_0}{2}}\left[\widetilde{A_{\infty}}\widetilde{B_{\infty}}\right] =\mathcal{E}_{\alpha}(U_{\alpha}\mu, U_{\alpha}\nu),\label{eq:main2-A}
\end{equation}
where we set \(\widetilde{A_t}:=\int_0^t e^{-\alpha s}dA_s\) and \(\widetilde{B_t}:=\int_0^t e^{-\alpha s}dB_s\), and \(U_{\alpha}\mu\) and \(U_{\alpha}\nu\) are \(\alpha\)-potentials of \(\mu\) and \(\nu\) respectively, for \(\alpha>0\). 
\end{proposition}

Here, \(m\) is the underlying measure of a regular Dirichlet form and \(\kappa\) is the killing measure appearing in the Beurling-Deny decomposition of a regular Dirichlet form. See Section \ref{prelim} for details. Moreover, \(\nu_0\) represents the energy arising from the process continuously hitting the cemetery point. More precisely, for a non-negative Borel measurable function \(f\), we define
 \[\int f d\nu_0 := \varlimsup_{t\searrow 0}\frac{1}{t}\int f(x) \mathbb{E}_x[e^{-2\zeta} {\bf 1}_{\{\partial \}}(X_{\zeta -}){\bf 1}_{\{\zeta \leq t\}}]dm(x).\] We remark that \(\nu_0\) is not a measure but a functional.  Nevertheless, to unify our treatment with \(m\) and \(\kappa\), and since \(\nu_0\) can be seen as the limit of measures \(\nu_t\) called entrance laws as \(t\searrow 0\), we adopt the notation \(\int fd\nu_0\). See Lemma \ref{nu_0Lemma} and Remark \ref{nu0rem} for details.

In Dirichlet form theory, the Revuz correspondence is characterized solely by the underlying measure \(m\). However, as Theorem \ref{main2} and Proposition \ref{main2-2} demonstrate, \(\kappa\) and \(\nu_0\) are essential to establishing the homeomorphism. See also Example \ref{ex4-2}, \ref{ex4-3}.

We also prove in Theorem \ref{submain1} the continuity of the Revuz map in the sense of \(L^{2-}(\mathbb{P}_x)\) with the local uniform topology for quasi-every \(x\). As another result of the continuity of the Revuz map, under the assumption to prevent immediate killing inside of stochastic processes (Assumption \ref{ass1}), we prove in Theorem \ref{mainvague} that if PCAFs converge almost surely and they are uniformly bounded in a certain sense, then corresponding smooth measures converge vaguely.

The organization of this paper is as follows. In Section \ref{prelim}, we present the preliminaries, including basic properties of Dirichlet forms, particularly PCAFs, Revuz measures, and their relations. In Section \ref{contimainsec}, we recall the topology introduced in \cite{NTTU} and present original results on the convergence of PCAFs. In Section \ref{mainsec}, we prove Theorem \ref{main2} and Proposition \ref{main2-2} by using the Fukushima decomposition of the potentials, the Beurling-Deny decomposition, and by capturing the energy of PCAFs. In Section \ref{mainsec2}, we prove vague convergence of Revuz measures under Assumption \ref{ass1}. In Section \ref{secexample}, we provide examples, including those showing that both \(\kappa\) and \(\nu_0\) are generally necessary for Theorem \ref{main2}.

\section{Preliminaries}\label{prelim}
This section is devoted to the preliminaries of Dirichlet form theory, in particular, PCAFs and Revuz measures. For more details on Dirichlet form theory, see \cite{CF,FOT, Os}.

Let \(E\) be a locally compact separable metric space and \(m\) be a positive Radon measure with \(\text{ supp}(m)=E\). The inner product in \(L^2(E;m)\) is denoted by \(\langle \cdot, \cdot \rangle_{m}\) and \(L^2\)-norm is denoted by \(||\cdot ||_{L^2(m)}\). For a non-negative definite symmetric bilinear form \(\mathcal{E}\) on a dense linear subspace \(\mathcal{F}\) of \(L^2(E;m)\), \((\mathcal{E}, \mathcal{F})\) is called a \textit{Dirichlet form} on \(L^2(E;m)\) if it is closed, that is, \(\mathcal{E}_1\)-complete where we define \(\mathcal{E}_{\alpha}(f,g):=\mathcal{E}(f,g)+\alpha \langle f,g\rangle_m\), and Markovian, that is, for any \(f\in \mathcal{F}\), it holds that \(g:=(0\vee f)\wedge 1 \in \mathcal{F}\) and \(\mathcal{E}(g)\leq \mathcal{E}(f)\). Here we defined \(a\wedge b:=\min\{a,b\}\) and \(a\vee b:=\max\{a,b\}\) for \(a,b\in \mathbb{R}\), and \(\mathcal{E}(f):=\mathcal{E}(f,f)\) for simplicity.

We set \(C_c(E)\) the family of all continuous functions with compact support and \(C_c(E)\) is equipped with a sup norm \(||\cdot||_{\infty}\). A Dirichlet form \((\mathcal{E},\mathcal{F})\) is called \textit{regular} if \(\mathcal{F} \cap C_c(E)\) is \(\mathcal{E}_1\)-dense in \(\mathcal{F}\) and \(||\cdot||_{\infty}\)-dense in \(C_c(E)\). For a regular Dirichlet form \((\mathcal{E},\mathcal{F})\) on \(L^2(E;m)\), there exists an \(m\)-symmetric Hunt process \(X=(\{X_t\}_{t\geq 0}, \Omega, \{\mathcal{F}_t\}_t, \{\mathbb{P}_x\}_{x\in E}, \zeta, \partial, \{\theta_t\}_{t\geq 0})\) on \(E\) associated with \((\mathcal{E},\mathcal{F})\). More precisely, \(\{X_t\}_t\) is an \(E\)-valued strong Markov process with quasi-left continuity on a probability space \(\Omega\) with a filtration \(\{\mathcal{F}_t\}_t\), probability measures \(\mathbb{P}_x\) with starting points \(x\in E\), a life time \(\zeta\), the cemetery point \(\partial\) and a shift operator \(\{\theta_t\}_t\), that is, \(\theta_t :\Omega \to \Omega\) satisfying \(X_s(\theta_t \omega)=X_{s+t}(\omega)\) for any \(\omega \in \Omega\) and \(s,t \geq 0\). The transition semigroup is defined by \(P_t f:=\mathbb{E}_x[f(X_t)]\) for a Borel function \(f\) and \(t\geq 0\), then \(P_t\) is \(m\)-symmetric, that is, \(\langle P_tf, g\rangle _m =\langle f, P_tg\rangle _m\) for \(f, g\in L^2(E;m)\), and \((\mathcal{E}, \mathcal{F})\) is associated with \(X\) in the following sense.
\begin{equation*}
\left\{ \begin{split}\mathcal{E}(f,g)&=\lim_{t\searrow 0}\frac{1}{t}\langle f-P_tf, g\rangle_m \text{\ \ for\ } f, g\in \mathcal{F},&\\
\mathcal{F}&=\{f\in L^2(E;m)\ |\ \lim_{t\searrow 0}\frac{1}{t}\langle f-P_tf, f\rangle_m<\infty\}.&
\end{split}\right.
\end{equation*}
 We define the resolvents \(\{R_{\alpha}\}_{\alpha>0}\) by
\[R_{\alpha}f(x):=\int_0^{\infty}e^{-\alpha t}P_tf(x)dt\]
for \(f\in L^2(m)\) and \(\alpha>0.\) Moreover, we define an \textit{extended Dirichlet space} \(\mathcal{F}_e\) by the totality of \(m\)-equivalence classes of all \(m\)-measurable functions \(f\) on \(E\) such that \(|f|<\infty, m\)-almost everywhere and there exists an \(\mathcal{E}\)-Cauchy sequence
\(\{f_n\}_{n\geq 1}\subset \mathcal{F}\) such that \(\lim_{n\to \infty}f_n=f\) \(m\)-almost everywhere on \(E\).

Throughout this paper, we assume that \((\mathcal{E}, \mathcal{F})\) is a regular Dirichlet form on \(L^2(E;m)\) and \(X\) is a Hunt process associated with \((\mathcal{E, \mathcal{F}})\).

We present the following definitions to interpret the capacity of a process \(X\).
\begin{definition}[{\cite[Definition 1.2.7, 1.2.12, \S 2.3, \S 5.4, A.1.28, A.2.12]{CF}}]
\ \vspace{-3mm}\\
\begin{enumerate}[nosep]
\item An increasing sequence of closed sets \(\{F_k\}_{k\geq 1}\) of \(E\) is a \textit{nest} if \(\cup_{k\geq 1} \{f\in \mathcal{F} : f=0 \ m\text{-a.e.\  on\  }E \setminus F_k\}\) is \(\mathcal{E}_1\)-dense in \(\mathcal{F}\).
\item  \(N\subset E\) is \textit{\(\mathcal{E}\)-polar} if there exists a nest \(\{F_k\}_{k\geq 1}\) such that \(N\subset \cap_{k\geq 1} (E\setminus F_k)\).
\item A statement depending on \(x\in D \subset E\) holds \(\mathcal{E}\)-\textit{quasi-everywhere} (q.e. in abbreviation) on \(D\) if there exists an \(\mathcal{E}\)-polar set \(N\subset D\) such that the statement holds for \(x\in D\setminus N\).
\item A function \(f\) is \(\mathcal{E}\)-\textit{quasi-continuous} if there exists a nest \(\{F_k\}_{k\geq 1}\) such that the restriction of \(f\) to \(F_k\) is finite and continuous on \(F_k\) for each \(k\geq 1\).
\item  A subset \(B\subset E\) is a \textit{nearly Borel set} if, for any probability measure \(\mu\) on \(E\cup \{\partial\}\), there exist Borel sets \(B_1, B_2\) such that \(B_1\subset B \subset B_2\) and \(\mathbb{P}_{\mu}(X_t \in B_2\setminus B_1 \text{\ for\ some\ }t\geq 0)=0\).
\item A subset \(N\subset E\) is \(m\)\textit{-inessential} if \(N\) is an \(m\)-negligible nearly Borel set such that \(\mathbb{P}_x(\sigma_{N}<\infty)=0\) for \(x\in E\setminus N,\) where \(\sigma_{N}:=\inf\{t>0;X_t\in N\}\) is the first hitting time to \(N\).

\item For an open set \(A\subset E,\) we define \(\text{Cap}_1(A):=\inf\{\mathcal{E}_1(f):f\in \mathcal{F}, f\geq 1\ m{\text -a.e.\  on}\ A\}\) where we define the infimum of the empty set to be \(\infty.\) For any set \(B\subset E,\) we define \(\text{Cap}_1(B):=\inf\{\text{Cap}_1(A): A \text{\ is\ an\ open\ set\ satisfying\ }B\subset A\}\). We call \(\text{Cap}_1(B)\) the \textit{capacity} of \(B\).
\item A subset \(N\subset E\) is \textit{Cap\(_1\)-polar} if Cap\(_1(N)=0.\)
\item A universally measurable non-negative function \(f\) defined q.e. on \(E\) is called an \textit{excessive function} for \(\{P_t\}_t\) if \(P_tf \nearrow f\) q.e. as \(t\searrow 0\). An excessive function \(f\) is called \textit{purely excessive} if \(\lim_{t\to \infty}P_tf(x)=0\) q.e. \(x\in E.\) An excessive function for \(\{e^{-\alpha t}P_t\}_t\) is called \textit{\(\alpha\)-excessive} for \(\alpha \geq 0\).

\end{enumerate}
\end{definition}
We remark that, if \((\mathcal{E}, \mathcal{F})\) is regular, a subset of \(E\) is \(\mathcal{E}\)-polar if and only if Cap\(_1\)-polar, an increasing sequence \(\{F_k\}_{k\geq 1}\) of closed sets is an \(\mathcal{E}\)-nest if and only if \( \lim_{k\to \infty} \rm{Cap}_1(K\setminus F_k)=0\) for any compact set \(K\subset E\), and all functions belonging to \(\mathcal{F}\) have quasi-continuous versions. For \(f\in \mathcal{F}\), denote by \(\widetilde{f}\) a quasi-continuous version of \(f.\)  Since we deal with a fixed regular Dirichlet form \((\mathcal{E}, \mathcal{F})\), for convenience, we drop \(``\mathcal{E}\)-" from the terminology such as \(\mathcal{E}\)-nest, \(\mathcal{E}\)-polar, \(\mathcal{E}\)-quasi-everywhere and \(\mathcal{E}\)-quasi-continuous. See \cite[Theorem 1.3.14, Lemma 1.3.15, \S 2.3]{CF} for details.

Next, we describe PCAFs, smooth measures and their correspondence.
\begin{definition}[{\cite[Definition A.3.1]{CF}}]\label{DefPCAF}
A \([-\infty,\infty]\)-valued stochastic process \(A=\{A_t\}_{t\geq 0}\) is called an \textit{additive functional} of \(X\) if there exist \(\Lambda \in \mathcal{F}_{\infty}\) and an \(m\)-inessential set \(N\subset E\) such that \(\mathbb{P}_x(\Lambda)=1\) for \(x\in E\setminus N\) and \(\theta_t \Lambda \subset \Lambda\) for any \(t>0\), and the following conditions hold.\\
\((A.1)\) For each \(t\geq 0,\) \(A_t|_{\Lambda}\) is \(\mathcal{F}_t|_{\Lambda}\)-measurable.\\
\((A.2)\) For any \(\omega \in \Lambda\), \(A_{\cdot}(\omega)\) is right continuous on \([0,\infty)\) and has left limits on \((0, \zeta(\omega))\), \(A_0(\omega)=0,\) \(|A_{t}(\omega)|<\infty\) for \(t<\zeta(\omega)\) and \(A_t(\omega)=A_{\zeta(\omega)}(\omega)\) for \(t\geq \zeta (\omega)\). Moreover the additivity condition
\[A_{t+s}(\omega)=A_t(\omega)+A_s(\theta_t\omega)\ \ \text{for\ every\ }t,s\geq 0,\]
is satisfied.

An additive functional \(A\) is called a \textit{positive continuous additive functional} (PCAF in abbreviation) if \(A\) is a \([0,\infty]\)-valued continuous process, and denote by \({\bf A}_c^+\) the family of all PCAFs.
\end{definition}
The set \(\Lambda\) appearing in Definition \ref{DefPCAF} is called the defining set of \(A\). A PCAF \(A\) is called a PCAF in the strict sense if \(N\) appearing in Definition \ref{DefPCAF} is empty.

PCAFs \(A\) and \(B\) are called \textit{\(m\)-equivalent} if \(\int_E\mathbb{P}_x(A_t\neq B_t)dm(x)=0\) for any \(t>0\). This is a condition equivalent to the existence of a common defining set \(\Lambda\) and a common \(m\)-inessential set \(N\) such that \(A_t(\omega)=B_t(\omega)\) for any \(t \geq 0\) and \(\omega \in \Lambda.\) 

\begin{definition}[{\cite[Definition 2.3.13]{CF}}]
A positive Borel measure \(\mu\) on \(E\) is a \textit{smooth measure} if \(\mu\) charges no \(\mathcal{E}\)-polar set and, there exists a nest \(\{F_k\}_k\) such that \(\mu(F_k)<\infty\) for every \(k\geq 1.\) Denote  by \(\mathcal{S}\) the family of all smooth measures.
\end{definition}

For a positive measure \(\nu\) on \(E\), \(\mathbb{E}_{\nu}\) (resp. \(\mathbb{P}_{\nu}\)) denotes \(\int_E \mathbb{E}_x[\ \cdot\ ] d\nu(x)\) (resp. \(\int_E \mathbb{P}_x(\ \cdot\ )d\nu(x)\)).

PCAFs and smooth measures correspond one-to-one in the following sense, and this correspondence is called \textit{the Revuz correspondence}. Therefore, a smooth measure is also called a \textit{Revuz measure}. The map from \(\mathcal{S}\) to \({\bf A}_c^+\) defined by this correspondence is called \textit{the Revuz map}.
\begin{theorem}[{\cite[Theorem 4.1.1]{CF}}]
\((i)\) For a PCAF \(A\), there exists a unique smooth measure \(\mu\) such that
\begin{equation}
\int_E fd\mu = \lim_{t \to 0}\frac{1}{t} \mathbb{E}_m\left[ \int_0^t f(X_s) dA_s  \right] \label{eq:AppPCAF-1}
\end{equation}
for any non-negative Borel function \(f\) on \(E.\)\\
\((ii)\) For any smooth measure \(\mu\), there exists a PCAF \(A\) satisfying \((\ref{eq:AppPCAF-1})\) up to the \(m\)-equivalence.
\end{theorem}

For example, for a bounded positive Borel function \(f\), we set
\[A_t:=\int_0^t f(X_s)ds,\]
then \(A:=\{A_t\}_t\) is a PCAF in the strict sense and \(A\) corresponds to a smooth measure \(fdm\). As another example, when a capacity of \(x_{0}\in E\) is positive, the local time \(L^{x_0}\) is a PCAF and its corresponding smooth measure is a Dirac measure \(\delta_{x_0}\).

 We remark that the bijection of the Revuz correspondence is characterized only by the underlying measure \(m\) for a regular Dirichlet form. On the other hand, in more general settings for right processes, \(m\) and its potential part are needed for the characterization. See \cite[(5.19), (6.29)]{FG} for example.

Next we prepare the basic definitions for the Fukushima decomposition. An additive functional \(M\) is called a \textit{martingale additive functional} (MAF in abbreviation) if \(\mathbb{E}_x[M_t^2]<\infty\) and \(\mathbb{E}_x[M_t]=0\) for each \(t>0\) and q.e. \(x\in E.\) We can check that an MAF \(M\) is \(\mathbb{P}_x\)-martingale for q.e. \(x\in E,\) so there exists \(\langle M \rangle \in {\bf A}_c^+\) satisfying \(\mathbb{E}_x[M_t^2]=\mathbb{E}_x[\langle M \rangle_t]\) for any \(t>0\) and q.e. \(x\in E.\) For an additive functional \(A\), the energy \(e(A)\) of \(A\) is defined by
\[e(A):=\lim_{t\searrow 0}\frac{1}{2t}\mathbb{E}_m[A^2_t].\]
A continuous additive functional \(N\) is called a \textit{continuous additive functional of zero energy} if \(\mathbb{E}_x[|N_t|]<\infty\) for q.e. \(x\in E\) and each \(t>0\), and \(e(N)=0.\)

The following decomposition is called \textit{the Fukushima decomposition}.
\begin{theorem}[{\cite[Theorem 4.2.6]{CF}}]\label{FukuDecomp}
For any \(u\in \mathcal{F}_e\), there exists an MAF \(M^{[u]}\) and a continuous additive functional of zero energy \(N^{[u]}\) uniquely such that
\[\widetilde{u}(X_t)-\widetilde{u}(X_0)=M^{[u]}+N^{[u]}, \mathbb{P}_x \text{-a.s.\ for\ }t\geq 0 \text{\ and\ q.e.\ }x\in E.\]
\end{theorem}

The following decomposition is called \textit{the Beurling-Deny decomposition}.
\begin{theorem}[{\cite[Section 4.3]{CF}}]\label{BDdecomp}
For \(f,g\in \mathcal{F}\), 
\[\mathcal{E}(f,g)=\mathcal{E}^{(c)}(f,g)+\frac{1}{2}\int_{E\times E \setminus d} (\widetilde{f}(x)-\widetilde{f}(y)) (\widetilde{g}(x)-\widetilde{g}(y))J(dxdy)+\int_E \widetilde{f}\widetilde{g}d\kappa.\]
Here \(\mathcal{E}^{(c)}\) is a strongly local Dirichlet form, that is a Dirichlet form satisfying \(\mathcal{E}^{(c)}(f,g)=0\) whenever \(f\in \mathcal{F}\) has a compact support and \(g\in \mathcal{F}\) is a constant on a neighbourhood of a support of \(f\), \(J\) is a symmetric Radon measure on \((E\times E) \setminus d\), where \(d\) denotes the diagonal set, and \(\kappa\) is a Radon measure on \(E\).
\end{theorem}

The Radon measure \(\kappa\) appearing in the Beurling-Deny decomposition is called \textit{the killing measure} of \(X\), and \(\kappa\) is a smooth measure (\cite[p.145]{CF}). The killing measure \(\kappa\) is characterized by the L\'{e}vy system, see \cite[Section A.3.4]{CF} for details.

\section{Convergence of PCAFs of finite energy integrals for quasi-every starting point}\label{contimainsec}
To consider the relationship between convergence of Revuz measures and PCAFs, we describe the topology for smooth measures of finite energy integrals introduced by Nishimori, Tomisaki, Tsuchida and Uemura \cite{NTTU}, and we state new results concerning the continuity of PCAFs in the sense of \(L^{2-}(\mathbb{P}_x)\) for q.e. \(x\in E\). Let \((\mathcal{E}, \mathcal{F})\) be a regular Dirichlet form on \(L^2(E;m)\) associated with an \(m\)-symmetric Hunt process \(X\).

\begin{definition}[{\cite[p.80]{CF}}]
A positive Radon measure \(\mu\) on \(E\) is called a measure of \textit{finite energy integral} if there exists \(C_{\mu}>0\) such that, for any \(g\in \mathcal{F}\cap C_c(E)\),
\[\int_E |g|d\mu \leq C_{\mu}\sqrt{\mathcal{E}_1(g)}.\]

Denote by \(\mathcal{S}_0\) the family of all measures of finite energy integrals.
\end{definition}
We remark that \(\mathcal{S}_0\subset \mathcal{S}\). For \(\mu \in \mathcal{S}_0\) and \(\alpha>0\), by the Riesz representation theorem, there exists a unique function \(U_{\alpha}\mu\), which is called \textit{an \(\alpha\)-potential}, such that \(\mathcal{E}_{\alpha}(U_{\alpha}\mu, g)=\int_Eg d\mu\) for any \(g\in \mathcal{F}\cap C_c(E).\) Moreover, for \(A\in {\bf A}_c^+\), \(\alpha>0\) and a bounded Borel function \(f\), we define \(U^{\alpha}_Af\) by
\begin{eqnarray}
U^{\alpha}_Af(x):=\mathbb{E}_x\left[\int_0^{\infty} e^{-\alpha t}f(X_t)dA_t \right] \label{eq:U_A}
\end{eqnarray}
for q.e. \(x\in E.\) By \cite[Lemma 5.1.3]{FOT}, \(U^{\alpha}_Af\) is a quasi-continuous version of \(U_{\alpha}(f\mu)\). See \cite[IV, \S 2]{BG}, \cite[\S 2.3]{CF}, \cite[\S 2.2]{FOT} for details of \(\alpha\)-potentials.

We define subclass \(\mathcal{S}_{00}\) of \(\mathcal{S}_{0}\) by
\[ \mathcal{S}_{00}:=\{\mu \in \mathcal{S}_0 : \mu(E)<\infty, ||U_1\mu||_{\infty}<\infty \}.\]
We remark that, for \(\mu \in \mathcal{S}_{00}\) and \(\alpha>0\), \(||U_{\alpha}\mu||_{\infty}<\infty\) follows from the resolvent equation \(U_{\alpha}\mu-U_{\beta}\mu+(\alpha-\beta)R_{\alpha}U_{\beta}\mu=0\).

The following metric on \(\mathcal{S}_0\) is introduced by \cite{NTTU}, which is stronger than the vague topology, in order to study the relationship between convergences of PCAFs and Revuz measures.

\begin{definition}[\cite{NTTU}]
Define a metric \(\rho\) on \(\mathcal{S}_0\) by \(\rho(\mu, \nu):=\sqrt{\mathcal{E}_1(U_1\mu -U_1\nu)}\) for \(\mu, \nu \in \mathcal{S}_0\). 
\end{definition}

\begin{proposition}[{\cite[Lemma 3.1, 3.4, 3.5]{NTTU}}]
The metric space \((\mathcal{S}_0, \rho)\) is complete and separable.
\end{proposition}

\begin{proposition}[{\cite[Proposition 3.8]{NTTU}}]
Let \(\mu_n, \mu \in \mathcal{S}_0\). If \(\rho(\mu_n,\mu)\to 0\) as \(n\to \infty\), then \(\mu_n\) converges to \(\mu\) vaguely, that is, for any \(f\in C_c(E),\) \(\lim_{n\to \infty}\int_E fd\mu_n=\int_E fd\mu.\)
\end{proposition}

\begin{remark}
\((\mathcal{F}, \mathcal{E}_1)\) is a Hilbert space and, for any \(\mu \in \mathcal{S}_0,\) \(\Phi_{\mu}:\mathcal{F}\ni f \mapsto \int \widetilde{f}d\mu \in \mathbb{R}\) is a bounded linear functional. The operator norm is equal to \(\rho\), that is,
\[\sup_{f\in \mathcal{F},\  \mathcal{E}_1(f) \leq 1}\left|\int_E f d\mu - \int_E f d\nu \right| = \rho(\mu, \nu).\]
\end{remark}

The following theorem is one of the results in \cite{NTTU}, which guarantees that subsequential almost sure convergence of PCAFs follows from the convergence of Revuz measures in \(\mathcal{S}_0\) with respect to \(\rho\).
\begin{theorem}[{\cite[Theorem 4.1]{NTTU}}]\label{NTTUmain}
For \(\mu_n, \mu \in \mathcal{S}_0\), let \(A^n, A \in {\bf A}_c^+\) be the corresponding PCAFs, respectively. If \(\rho(\mu_n,\mu)\to \infty\) as \(n\to 0\), then there exists a subsequence \(\{n_k\}\) such that \(A^{n_k}\) converges to \(A\) \(\mathbb{P}_x\)-almost surely for q.e. \(x\in E\) with the local uniform topology, that is, for q.e. \(x\in E\), \[\mathbb{P}_x\left(\lim_{n_k\to \infty}\sup_{0\leq t\leq T}|A_t^{n_k}-A_t|=0 \text{\ for\ any\ }T \right)=1.\]
\end{theorem}
We remark that, in Theorem \ref{NTTUmain} (\cite[Theorem 4.1]{NTTU}), a subsequence \(\{n_k\}\) does not depend on q.e. \(x\in E\).

At the end of this section, we state new results on the convergence of PCAFs. The following result makes a slightly stronger claim than Theorem \ref{NTTUmain} (\cite[Theorem 4.1]{NTTU}).
\begin{theorem}\label{submain1}
For \(\mu_n, \mu \in \mathcal{S}_0\), let \(A^n, A \in {\bf A}_c^+\) be the corresponding PCAFs, respectively. If \(\rho(\mu_n,\mu)\to 0\) as \(n\to \infty\), then \(A^n\) converges to \(A\) with the local uniform topology in \(L^1(\mathbb{P}_x)\) for q.e. \(x\in E\), that is, for any \(T>0\) and q.e. \(x\in E\), \[\lim_{n\to \infty}\mathbb{E}_x\left[\sup_{0\leq t\leq T}|A_t^n-A_t|\right]=0.\]
\end{theorem}

Our proof of Theorem \ref{submain1} is based on the following Fukushima decomposition of potentials.
\begin{lemma}[{\cite[Lemma 5.4.1]{FOT}}]\label{FukuDecomp1}
For \(\mu \in \mathcal{S}_0\) corresponding to \(A\in {\bf A}_c^+\) and \(\alpha >0\), there exists an MAF \(M^{[U_{\alpha}\mu]}\) such that, for \(t\geq 0\) and q.e. \(x\in E\),
\[\widetilde{U_{\alpha}\mu}(X_t)-\widetilde{U_{\alpha}\mu}(X_0)=M_t^{[U_{\alpha}\mu]}+\alpha \int_0^t \widetilde{U_{\alpha}\mu}(X_s)ds-A_t, \ \mathbb{P}_x \text{-a.s.},\]
and \(\alpha \int_0^t \widetilde{U_{\alpha}\mu}(X_s)ds-A_t\) is a continuous additive functional of zero energy. Here, \(\widetilde{U_{\alpha}\mu}\) is a quasi-continuous version of \(U_{\alpha}\mu\).
\end{lemma}

\begin{proof}[Proof of Theorem \ref{submain1}]
By \cite[Corollary 2.3.11]{CF}, it is enough to show that \[\lim_{n\to \infty}\mathbb{E}_{\nu}\left[\sup_{0\leq t\leq T}|A_t^n-A_t|\right]=0\]
for any \(\nu \in \mathcal{S}_{00}.\) By the Fukushima decomposition (Lemma \ref{FukuDecomp1}), there exists a martingale additive functional \(M^{[U_1\mu_n]}\) such that,
\begin{equation}\label{eq:NTTU1}
\widetilde{U_1\mu_n}(X_t)-\widetilde{U_1\mu_n}(X_0)=M_t^{[U_1\mu_n]}+\int_0^t \widetilde{U_1\mu_n}(X_s)ds-A_t^n,
\end{equation}
\(\mathbb{P}_x\)-almost surely for q.e. \(x\in E\). So, at first, we prove that \(A^n\) converges to \(A\) in measure \(\mathbb{P}_{\nu}\) for any \(\nu \in \mathcal{S}_{00}\) with the local uniform topology by checking convergences of each term in \((\ref{eq:NTTU1})\) except the term of \(A^n\). Since \(\mathbb{P}_{\nu}\) is a finite measure for \(\nu \in \mathcal{S}_{00}\), for convenience, we use the term of the convergence in probability \(\mathbb{P}_{\nu}\) to mean the convergence in measure \(\mathbb{P}_{\nu}\) for \(\nu \in \mathcal{S}_{00}\).

For \(\nu \in \mathcal{S}_{00}\) and \(T>0\), by Doob's inequality and \cite[Lemma 4.2.4]{CF}, we have
\begin{eqnarray*}
\mathbb{E}_{\nu}\left[\sup_{0\leq t\leq T}|M_t^{[U_1\mu_n]}-M_t^{[U_1\mu]}|^2\right] &\leq & 4  \mathbb{E}_{\nu}\left[|M_T^{[U_1\mu_n]}-M_T^{[U_1\mu]}|^2\right]\\
&= & 4  \mathbb{E}_{\nu}\left[\langle M^{[U_1\mu_n]}-M^{[U_1\mu]} \rangle_T\right]\\
&\leq & 8 (1+T)||U_1\nu ||_{\infty} \mathcal{E}(U_1\mu_n - U_1\mu)\\
&\leq & 8 (1+T)||U_1\nu ||_{\infty} \mathcal{E}_1(U_1\mu_n - U_1\mu).
\end{eqnarray*}
So \(\sup_{0\leq t\leq T}|M_t^{[U_1\mu_n]}-M_t^{[U_1\mu]}|\) converges to \(0\) in \(L^2(\mathbb{P}_{\nu})\) and so in probability \(\mathbb{P}_{\nu}\) as \(n\) tends to infinity.

For \(\nu \in \mathcal{S}_{00}\) and \(T>0\), we have
\begin{eqnarray*}
\lefteqn{\mathbb{E}_{\nu}\left[\sup_{0\leq t\leq T}\left|\int_0^t \widetilde{U_1\mu_n}(X_s)ds-\int_0^t \widetilde{U_1\mu}(X_s)ds\right|^2\right]
\leq  T \int_0^T \mathbb{E}_{\nu}\left[\left| U_1\mu_n(X_s)- U_1\mu(X_s)\right|^2 \right]ds}\\
&= & T \lim_{k\to \infty} \int_0^T \mathbb{E}_{\nu}\left[\left| U_1\mu_n(X_s)- U_1\mu(X_s)\right|^2 {\bf 1}_{\{|U_1\mu_n-U_1\mu|\leq k\}} \right]ds\\
&= & T \lim_{k\to \infty}\int_0^T\int_E P_s \left(\left| U_1\mu_n - U_1\mu \right|^2{\bf 1}_{\{|U_1\mu_n-U_1\mu|\leq k\}}(X_s)\right)(x)d\nu(x)ds\\
&\leq & T e^T \lim_{k\to \infty}\int_E R_1 \left(\left| U_1\mu_n - U_1\mu \right|^2{\bf 1}_{\{|U_1\mu_n-U_1\mu|\leq k\}}\right)(x) d\nu(x)\\
&= & T e^T\lim_{k\to \infty}\mathcal{E}_1\left(U_1 \nu, R_1 \left(\left| U_1\mu_n - U_1\mu \right|^2{\bf 1}_{\{|U_1\mu_n-U_1\mu|\leq k\}}\right)\right)\\
&= &  T e^T \lim_{k\to \infty}\int_E U_1\nu(x) \left| U_1\mu_n - U_1\mu \right|^2{\bf 1}_{\{|U_1\mu_n-U_1\mu|\leq k\}}(x) dm(x)\\
&= & T e^T \int_E U_1\nu(x) \left| U_1\mu_n - U_1\mu \right|^2(x) dm(x)\\
&\leq &  T e^T ||U_1\nu ||_{\infty} \int_E \left| U_1\mu_n - U_1\mu \right|^2(x) dm(x)\\
&\leq &  T e^T ||U_1\nu ||_{\infty} \mathcal{E}_1\left(U_1\mu_n - U_1\mu\right).
\end{eqnarray*}
So \(\sup_{0\leq t\leq T}\left|\int_0^t \widetilde{U_1\mu_n}(X_s)ds-\int_0^t \widetilde{U_1\mu}(X_s)ds\right|\) converges to \(0\) in \(L^2(\mathbb{P}_{\nu})\) and so in probability \(\mathbb{P}_{\nu}\) as \(n\) tends to infinity.

For any \(\nu \in \mathcal{S}_{00}\) and \(T>0\), we have
\begin{eqnarray*}
\mathbb{E}_{\nu}\left[\sup_{0\leq t\leq T}\left|\widetilde{U_1\mu_n}(X_0)- \widetilde{U_1\mu}(X_0)\right|\right]&=&\mathbb{E}_{\nu}\left[\left|\widetilde{U_1\mu_n}(X_0)-\widetilde{U_1\mu}(X_0)\right|\right]\\
&=& \mathcal{E}_1 (U_1\nu, |U_1\mu_n-U_1\mu|)\\
&\leq & \sqrt{\mathcal{E}_1(U_1 \nu)} \sqrt{\mathcal{E}_1(U_1\mu_n - U_1\mu)}.
\end{eqnarray*}

So \(\sup_{0\leq t\leq T}\left|\widetilde{U_1\mu_n}(X_0)- \widetilde{U_1\mu}(X_0)\right|\) converges to \(0\) in \(L^1(\mathbb{P}_{\nu})\) and so in probability \(\mathbb{P}_{\nu}\) as \(n\) tends to infinity.

By \cite[Lemma 5.1.1]{FOT}, for any \(\nu \in \mathcal{S}_{00}\) and \( T, \delta>0\), we have
\[\mathbb{P}_{\nu}\left(\sup_{0\leq t\leq T}\left|\widetilde{U_1\mu_n}(X_t) - \widetilde{U_1\mu}(X_t) \right| \geq \delta \right)\leq \frac{e^T}{\delta} \sqrt{\mathcal{E}_1(U_1 \nu)} \sqrt{\mathcal{E}_1(U_1\mu_n - U_1\mu)}.\]
So \(\sup_{0\leq t\leq T}\left|\widetilde{U_1\mu_n}(X_t) - \widetilde{U_1\mu}(X_t) \right|\) converges to \(0\) in probability \(\mathbb{P_{\nu}}\).

Hence, by (\ref{eq:NTTU1}), \(\sup_{0\leq t\leq T}\left|A_t^n-A_t \right|\) converges to \(0\) in probability \(\mathbb{P_{\nu}}\) for any \(\nu\in \mathcal{S}_{00}\). We next prove the uniform integrability of a sequence \(\{\sup_{0\leq t\leq T}|A^n_t-A_t|\}_n\) under \(\mathbb{P}_{\nu}\) for any \(\nu \in \mathcal{S}_{00}.\) For any \(\nu \in \mathcal{S}_{00}\), we have
\begin{eqnarray*}
\mathbb{E}_{\nu}\left[\left( \sup_{0\leq t\leq T}|A^n_t-A_t|\right)^2\right] &\leq & 2\mathbb{E}_{\nu}\left[\sup_{0\leq t\leq T}\left((A^n_t)^2+(A_t)^2\right)\right]\\
&=& 2\mathbb{E}_{\nu}\left[ \left((A^n_T)^2+(A_T)^2\right)\right].\end{eqnarray*}

By \cite[Excercise 4.1.7] {CF}, it holds that \(\mathbb{E}_{\nu}[(\widetilde{A^n_{\infty}})^2]=2\int_E U^2_{A^n}U^1_{A^n}1 d\nu \) where \(\widetilde{A_T^n}:=\int_0^{T} e^{-t}dA_t^n\). Since \(U^{\alpha}_{A^n}f=\widetilde{U_{\alpha}(f\mu_n)}\) for any bounded non-negative Borel measurable function \(f\) by \cite[Lemma 5.1.3]{FOT}, and an inequality \(A_T^n \leq e^T \widetilde{A_T^n} \leq e^T \widetilde{A_{\infty}^n},\) we have
\begin{eqnarray*}
2\mathbb{E}_{\nu}\left[ \left((A^n_T)^2+(A_T)^2\right)\right]&\leq &4e^{2T}\int_E U^2_{A^n}U^1_{A^n}1 d\nu + 4e^{2T}\int_E U^2_{A}U^1_{A}1 d\nu \\
&=& 4e^{2T} \lim_{k\to \infty} \left\{\int_E U^2_{A^n}(U^1_{A^n}1 \wedge k) d\nu  +\int_E U^2_{A}(U^1_{A}1 \wedge k) d\nu    \right\}\\
&=& 4e^{2T} \lim_{k\to \infty} \left\{\mathcal{E}_2(U_2\nu, U_2((\widetilde{U_1 \mu _n}\wedge k)\mu _n)) +U_2((\widetilde{U_1 \mu}\wedge k)\mu)) \right\}\\
&=& 4e^{2T} \lim_{k\to \infty} \left\{\int_E U_2\nu (\widetilde{U_1 \mu _n}\wedge k)d\mu _n +\int_E U_2\nu (\widetilde{U_1 \mu}\wedge k) d\mu) \right\}\\
&=& 4e^{2T}  \left\{\int_E U_2\nu \widetilde{U_1 \mu _n}d\mu _n +\int_E U_2\nu \widetilde{U_1 \mu} d\mu \right\}\\
&\leq & 4e^{2T} ||U_2\nu||_{\infty} (\mathcal{E}_1(U_1\mu _n)+\mathcal{E}_1(U_1\mu )).
\end{eqnarray*} 
By the convergence of \(U_1\mu _n\) to \(U_1\mu \) in \(\mathcal{E}_1\) and the above, \(\{\sup_{0\leq t\leq T}|A^n_t-A_t|\}_n\) under \(\mathbb{P}_{\nu}\) is uniformly integrable. Thus \(\sup_{0\leq t\leq T}\left|A_t^n-A_t \right|\) converges to \(0\) in  \(L^1(\mathbb{P_{\nu}})\) for any \(\nu\in \mathcal{S}_{00}\) and so does in \(L^1(\mathbb{P}_x)\) for q.e. \(x\in E.\)
\end{proof}
We remark that we cannot replace the convergence of \(\mu_n\) to \(\mu\) in \(\rho\) with the vague convergence in Theorem \ref{submain1}. See Example \ref{ex1} for details.

\begin{remark}
In the proof of Theorem \ref{submain1}, we proved that, for any \(\nu \in \mathcal{S}_{00}\) and \(T, \delta>0\), there exists positive constant \(C_{\nu}\) depending on \(\nu\),
\[\mathbb{P}_{\nu}\left(\sup_{0\leq t\leq T}\left|A^n_t - A_t \right| \geq \delta \right)\leq C_{\nu} \sqrt{\mathcal{E}_1(U_1\mu_n - U_1\mu)}.\]
So, by using Borel-Cantelli's lemma, we can take a subsequence \(\{n_k\}_k\) independent of \(\nu\) such that
\[\mathbb{P}_x\left(\lim_{n_k\to \infty}\sup_{0\leq t\leq T}|A_t^{n_k}-A_t|=0 \text{\ for\ any\ }T \right)=1\]
for any \(T>0\) and q.e. \(x\in E\), so Theorem \ref{NTTUmain} (\cite[Theorem 4.1]{NTTU}) follows.
\end{remark}

\begin{remark}
Nishimori, Tomisaki, Tsuchida and Uemura proved Theorem \ref{NTTUmain} (\cite[Theorem 4.1]{NTTU}) by constructing an approximating sequence of PCAFs \(\{A^{n,k}\}_k\) whose corresponding Revuz measures are absolutely continuous with respect to \(m\) and mainly using properties of the resolvent and the diagonal argument. Our proof is based on the Fukushima decomposition for potentials \(U_1\mu_n\), so this is a different approach from theirs.
\end{remark}

\begin{remark}
In the proof of Theorem \ref{submain1}, we proved that \(\{\mathbb{E}_{\nu}[\sup_{t\leq T}|A_t^n-A_t|^2]\}_n\) is bounded for any \(\nu \in \mathcal{S}_{00}\).
By the H\"{o}lder inequality, the conclusion of Theorem \ref{submain1} can be strengthened to the convergence of \(A^n\) to \(A\) in \(L^p(\mathbb{P}_x)\) with the local uniform topology for q.e. \(x\in E\) and any \(p<2\).
\end{remark}

In a special case, we can show the \(L^2(\mathbb{P}_x)\) convergence of PCAFs as stated in the following corollary.

\begin{corollary}\label{submain2}
For \(\mu_n, \mu \in \mathcal{S}_{00}\), let \(A^n, A \in {\bf A}_c^+\) be the corresponding PCAFs, respectively. If \(\rho(\mu_n,\mu)\to 0\) as \(n\to \infty\) and \(\sup_{n}||U_1\mu_n||_{\infty}<\infty\), then \(A^n\) converges to \(A\) in \(L^2(\mathbb{P}_x)\) for q.e. \(x\in E\) with the local uniform topology, that is, for any \(T>0\) and q.e. \(x\in E\), \[\lim_{n\to \infty}\mathbb{E}_x\left[\sup_{0\leq t\leq T}|A_t^n-A_t|^2\right]=0.\]
\end{corollary}
\begin{proof}
By Theorem \ref{submain1}, \(A^n\) converges to \(A\) with the local uniform topology in \(L^1(\mathbb{P}_x)\) for q.e. \(x\in E\). In the proof of Theorem \ref{submain1}, we proved convergences of \(M^{[U_1\mu_n]}\) and \(\int_0^{\cdot} \widetilde{U_1\mu_n}(X_s)ds\) in \(L^2(\mathbb{P}_{\nu})\) with the local uniform topology, respectively, so it is enough to show that \(\widetilde{U_1\mu_n}(X)\) converges in \(L^2(\mathbb{P}_{\nu})\) with the local uniform topology. By the uniform boundedness of \(||U_1\mu_n||_{\infty}\), we set \(C:=\sup_n ||U_1\mu_n||_{\infty}\), so we have
\[\mathbb{E}_{\nu}\left[\sup_{0\leq t\leq T}|\widetilde{U_1\mu_n}(X_t)-\widetilde{U_1\mu}(X_t)|^2 \right]\leq 2C \mathbb{E}_{\nu}\left[\sup_{0\leq t\leq T}|\widetilde{U_1\mu_n}(X_t)-\widetilde{U_1\mu}(X_t)| \right]\]
for any \(T\geq 0\) and q.e. \(x\in E\). So the proof is completed.
\end{proof}

\section{Bicontinuity of the Revuz map restricted to the space of smooth measures of finite energy integrals}\label{mainsec}
In this section, we prove Theorem \ref{main2}, that is, the Revuz map from \(\mathcal{S}_0\) to \({\bf A}_c^+\) is bicontinuous. We continue to use the metric \(\rho\) on \(\mathcal{S}_0\) described in the previous section.
We prove Theorem \ref{main2} by capturing the energy of PCAFs in Proposition \ref{main2-2}. Let \((\mathcal{E}, \mathcal{F})\) be a regular Dirichlet form on \(L^2(E;m)\) associated with an \(m\)-symmetric Hunt process \(X\), and \(\kappa\) be the killing measure of \(X\).

We set \(\varphi_{\alpha}(x):=\mathbb{E}_x[e^{-\alpha \zeta}{\bf 1}_{\{\partial\}}(X_{\zeta -})]\) and define
 \[\int f d\nu_0 := \varlimsup_{t\searrow 0}\frac{1}{t}\int f(x) \mathbb{E}_x[e^{-2 \zeta} {\bf 1}_{\{\partial \}}(X_{\zeta -}){\bf 1}_{\{\zeta \leq t\}}]dm(x) = \varlimsup_{t\searrow 0} \frac{1}{t}\int f(x) (\varphi_2(x)-e^{-2 t}P_t\varphi_{2}(x))dm(x).\] 
 for a positive Borel measurable function \(f\). We remark that, for a \(2\)-excessive function \(f\),
 \[\int f d\nu_0 = \lim_{t\searrow 0} \frac{1}{t}\int f(x) (\varphi_{2}(x)-e^{-2 t}P_t\varphi_{2}(x))dm(x).\]
is an increasing limit, and this is called the energy functional of \(f\) and \(\varphi_2\) for \(\{e^{-2t}P_t\}_t\). See \cite[\S 5.4]{CF} and \cite[\S 3]{G} for details. In our cases, we need to consider only the limit of the energy functional with \(\varphi_2\) is either zero or non-zero, so we extend the energy functional with \(\varphi_2\) for positive Borel functions. We remark that \(\nu_0\) is not a measure in general, but the following holds. Recall that \(\widetilde{A_t}:=\int_0^t e^{-\alpha s}dA_s.\)

\begin{lemma}\label{nu_0Lemma}
{\rm (i) (monotonicity)}  For positive Borel functions \(f, g\) satisfying \(f\leq g\) \(m\)-a.e., \[\int f d\nu_0 \leq \int g d\nu_0.\]
{\rm (ii) (subadditivity)} For positive Borel functions \(f, g\), \[\int (f+g) d\nu_0 \leq \int f d\nu_0 + \int g d\nu_0.\]
{\rm (iii)} For \(\alpha \geq 0\) and an \(\alpha\)-excessive function \(f\), 
\begin{equation}
\int fd\nu_0 =  \lim_{t\searrow 0}\frac{1}{t}\int f(x) (\varphi_{\alpha}(x)-e^{-\alpha t}P_t\varphi_{\alpha}(x))dm(x). \label{eq:excessive_indep}
\end{equation}
{\rm (iv)} For \(\mu, \nu \in \mathcal{S}_0\), let \(A, B \in {\bf A}_c^+\) be the corresponding PCAFs, respectively. Then \(\mathbb{E}_{\nu_0}[(\widetilde{A_{T}}-\widetilde{B_{T}})^2]\) and  \(\mathbb{E}_{\nu_0}[\sup_{t\leq T}|A_{t}-B_t|^2]\) are finite for any \(T\geq 0\). 
\end{lemma}
\begin{proof}
(i) and (ii) are obvious. We prove (iii). We set 
\[L^{(\alpha)}(f, g):= \lim_{t\searrow 0}\frac{1}{t}\int f(x) (g(x)-e^{-\alpha t}P_tg(x))dm(x)\]
for an \(\alpha\)-excessive function \(f\) and a purely \(\alpha\)-excessive function \(g\). By \cite[\S 5.4]{CF}, the right hand side of (\ref{eq:excessive_indep}) is well-defined. By the resolvent equation, that is, \(\varphi_{\alpha}-\varphi_2=(2-\alpha)R_2\varphi_{\alpha}\), we have
\begin{eqnarray*}
\lefteqn{\frac{1}{t}\int_E f\ (1-e^{-2t}P_t)\varphi_2\ dm - \frac{1}{t}\int_E f\ (1-e^{-\alpha t}P_t)\varphi_{\alpha}\ dm}\\
&=&\frac{\alpha -2}{t} \int_E f\ (1-e^{-\alpha t}P_t)R_2 \varphi_{\alpha}\ dm + \frac{1}{t} \int_E f\ (e^{-\alpha t}-e^{-2t})P_t \varphi_{\alpha}\ dm
\end{eqnarray*}
By taking the limit sup as \(t\) tends to \(0\), (iii) is proved.

We prove (iv). Since \(\mathbb{E}_x[(\widetilde{A_{T}}-\widetilde{B_{T}})^2]\leq 2 \mathbb{E}_x[(\widetilde{A_{T}})^2]+ 2\mathbb{E}_x[(\widetilde{B_{T}})^2]\), we may assume that \(B=0.\) Since \(\mu \in \mathcal{S}\), there exists a nest \(\{F_k\}_k\) such that \(U_{\alpha}({\bf 1}_{F_k}\mu)\) is bounded for each \(k\) by \cite[Theorem 2.3.15]{CF}. By the monotone convergence theorem, we have
\begin{eqnarray*}
\mathbb{E}_{\nu_0}[(\widetilde{A_{T}})^2] 
&\leq & e^{2T}  \int \mathbb{E}_x[(\widetilde{A_{\infty}})^2] d\nu_0\\
&=& e^{2T}  \varlimsup_{s\searrow 0} \lim_{k\to \infty}\frac{1}{s} \int \mathbb{E}_x[(\widetilde{{\bf 1}_{F_k}A_{\infty}})^2] (\varphi_{2\alpha} -e^{-2\alpha s}P_s \varphi_{2\alpha}) dm(x),
\end{eqnarray*}
where \(({\bf 1}_{F_k}A)_t:=\int_0^t {\bf 1}_{F_k}(X_s)dA_s\). Since \(U_{\alpha}({\bf 1}_{F_k}\mu_n)\) is bounded, \(\mathbb{E}_x[(\widetilde{{\bf 1}_{F_k}A_{\infty}})^2]=2U_{2\alpha}(U_{\alpha}({\bf 1}_{F_k}\mu){\bf 1}_{F_k}\mu)\) is \(2\alpha\)-excessive by \cite[Lemma 1.2.4]{CF}. Hence
\[ \frac{1}{s} \int \mathbb{E}_x[(\widetilde{{\bf 1}_{F_k}A_{\infty}})^2] (\varphi_{2\alpha} -e^{-2\alpha s}P_s \varphi_{2\alpha}) dm(x)\]
is increasing as \(s\searrow 0\) and \(k \nearrow \infty\) respectively. So we can exchange two limits, and, by \cite[Theorem 5.4.3(iv)]{CF}, we have
\begin{eqnarray*}
\varlimsup_{s\searrow 0} \lim_{k\to \infty}\frac{1}{s} \int \mathbb{E}_x[(\widetilde{{\bf 1}_{F_k}A_{\infty}})^2] (\varphi_{2\alpha} -e^{-2\alpha s}P_s \varphi_{2\alpha}) dm(x) &=& 2 \lim_{k\to \infty} \int_{E} U_{2\alpha}(U_{\alpha}({\bf 1}_{F_k}\mu){\bf 1}_{F_k}\mu)\ d\nu_0\\
&=& 2 \lim_{k\to \infty} \int_{F_k} U_{\alpha}({\bf 1}_{F_k}\mu)\ \varphi_{2\alpha} \ d\mu\\
&\leq & 2 \mathcal{E}_{\alpha}( U_{\alpha} \mu) \ <\ \infty.
\end{eqnarray*}
Similarly, \(\mathbb{E}_{\nu_0}[\sup_{t\leq T}|A_{t}^n-A_t|^2]\) is also finite.
\end{proof}

\begin{remark}\label{nu0rem}
{\rm (i)} By applying \cite[Lemma 2.19, 2.20]{T} to \(\{e^{-\alpha t}P_t\}_t,\) \(\varphi_{\alpha}\) belongs to the local Dirichlet space \(\mathcal{F}_{loc}\) and \(\int f d\nu_0 = 0\) for any positive \(f\in \mathcal{F}\cap C_c(E)\). By \cite[Proposition 4.2.3]{CF}, \(\int f^2 d\nu_0 = 0\) for any positive \(f\in \mathcal{F}_e\). These indicate that \(\nu_0\) is the energy for the part that the process \(X\) continuously escaping to the cemetery point \(\partial\).\\
{\rm (ii)} By \cite[Theorem 7.5.6]{CF}, if \(\mathbb{P}_x(X_{\zeta -}=\partial, \zeta <\infty)=\mathbb{P}_x(\zeta <\infty)\) and some technical assumptions hold, there exists a right process \(X^{ref}\) on \(E \cup \{\partial\}\) such that the part process of \(X^{ref}\) on \(E\) is \(X.\) In \cite{CF}, \(X^{ref}\) is called a one-point extension of \(X\). For example, the reflecting Brownian motion on \([0, \infty)\) is a one-point extension of the absorbing Brownian motion on \([0, \infty)\) (\cite[\S 7.6.(\(1^{\circ}\))]{CF}). In the case that a one-point extension exists, there exists a family of \(\sigma\)-finite measures \(\{\nu_t\}_{t>0}\) called an entrance law, satisfying \(\varphi_0 dm = \int_0^{\infty} \cdot\  d\nu_t\) and \(L^{(0)}(\varphi_0, f)=\lim_{t\searrow 0} \int_E fd\nu_t\) for a \(0\)-excessive function \(f\) by \cite[\S 5.7, p344]{CF}. 

Although \(\nu_0\) satisfy the subadditively and the monotonicity over non-negative Borel functions, it does not satisfy additivity in general, so \(\nu_0\) is not a measure. However, we adopt the notation \(\int fd\nu_0\) to unify our treatment with \(m\) and \(\kappa\), which are Radon measures, and because \(\nu_0\) arises naturally as a limiting object of entrance laws \(\nu_t\) and as a functional representing the energy of mass escaping to the cemetery point \(\partial\).

\end{remark}

Recall that \(\mathbb{E}_m\:=\int_E \mathbb{E}_{x}[\ \cdot\ ]dm(x)\) and  \(\mathbb{E}_{\kappa}\:=\int_E \mathbb{E}_{x}[\ \cdot\ ]d\kappa(x)\).
By considering the counterparts of potentials of \(m\), \(\kappa\) and \(\nu_0\), we prove Proposition \ref{main2-2} as follows.
\begin{proof}[Proof of Proposition \ref{main2-2}]
 By using the polarization identity, it is enough to consider the case of \(A=B\). 
Without loss of generality, we can assume that \(m\) is a \(\sigma\)-finite measure by using the quasi-homeomorphism method in \cite[Theorem 1.4.3 and Theorem 1.5.1]{CF}. We take \(\{E_j\}_j\) satisfying \(\bigcup_j E_j =E\) and \(m(E_j)<\infty\). By \cite[Exercise 4.1.7]{CF}, we have

\begin{eqnarray*}
\mathbb{E}_m\left[(\widetilde{A_{\infty}})^2 \right] &=&  \langle 1, 2U^{2\alpha}_{A}U^{\alpha}_{A}1 \rangle_m \\
&=& \lim_{j\to \infty}\langle {\bf 1}_{E_j}, 2U^{2\alpha}_{A}U^{\alpha}_{A}1 \rangle_m \\
&=& \lim_{j\to \infty} \mathcal{E}_{2\alpha}(R_{2\alpha}{\bf 1}_{E_j}, 2U^{2\alpha}_{A}U^{\alpha}_{A}1)\\
&=& \lim_{j\to \infty} 2\langle R_{2\alpha}{\bf 1}_{E_j}, U_{\alpha}\mu \rangle_{\mu} \\
&=& 2 \langle R_{2\alpha}1, U_{\alpha}\mu \rangle_{\mu} 
\end{eqnarray*}

Since \(\kappa \in \mathcal{S}\), there exists a nest \(\{F_l\}\) such that \({\bf 1}_{F_l}\kappa \in \mathcal{S}_{00}\). Let \(C\) be a PCAF corresponding to \(\kappa\). By the definition of the L\'{e}vy system (or the proof of \cite[Theorem 4.2.1]{CF}), for \(\alpha>0,\) we have 
\begin{eqnarray}
\nonumber \lim_{l\to \infty}\widetilde{U_{\alpha}({\bf 1}_{F_l}\kappa)}(x)&=& \lim_{l\to \infty}\mathbb{E}_x\left[\int_0^{\infty}e^{-\alpha t}{\bf 1}_{F_l}(X_t)dC_t \right]\\
\nonumber &=& \lim_{l\to \infty}\mathbb{E}_x\left[e^{-\alpha \zeta}{\bf 1}_{F_l}(X_{\zeta -}) \right]\\
 \nonumber &=& \mathbb{E}_x[e^{-\alpha \zeta}{\bf 1}_{E}(X_{\zeta -})]\\
&=& 1-\alpha R_{\alpha}1(x) -\mathbb{E}_x[e^{-\alpha \zeta}{\bf 1}_{\{\partial\}}(X_{\zeta -})].\label{eq:Ukappa}
\end{eqnarray}
By \cite[Exercise 4.1.7]{CF}, \cite[Theorem 5.1.1]{FOT} and (\ref{eq:Ukappa}), similarly, we have
\begin{eqnarray*}
\mathbb{E}_{\kappa}\left[(\widetilde{A_{\infty}})^2\right]&=& \lim_{l\to \infty}\langle {\bf 1}_{F_l}, 2U^{2\alpha}_{A}U^{\alpha}_{A}1 \rangle_{\kappa} \\
&=& 2 \int_E  (1-2\alpha R_{2\alpha}1 -\mathbb{E}_x[e^{-2\alpha \zeta}{\bf 1}_{\{\partial\}}(X_{\zeta})])\ U_{\alpha}\mu \  d\mu.
\end{eqnarray*}

 Since \(\mathbb{E}_{x}\left[(\widetilde{A_{\infty}})^2\right]=2U_A^{2\alpha}U_A^{\alpha} 1 \) and \(\varphi_{2\alpha}(x):=\mathbb{E}_x[e^{-2\alpha \zeta}{\bf 1}_{\{\partial\}}(X_{\zeta})] \) are purely excessive functions with respect to \(\{e^{-2\alpha}P_t\}_t\), \(\mathbb{E}_{\nu_0}\left[(\widetilde{A_{\infty}})^2\right]\) is well-defined by \cite[\S 5.4]{CF}. By  \cite[Theorem 5.4.3.(iv)]{CF}, we have
\begin{eqnarray*}
\mathbb{E}_{\nu_0}\left[(\widetilde{A_{\infty}})^2\right] &=& \lim_{t\searrow 0}\frac{1}{t}\langle 2U_A^{2\alpha}U_A^{\alpha} 1, (1-e^{-2\alpha t}P_t )\varphi_{2\alpha} \rangle_m\\
&=&2 \int_E   \varphi_{2\alpha} \ U_{\alpha}\mu \  d\mu.
\end{eqnarray*}

Hence we have

\begin{eqnarray*}
\mathbb{E}_{{\alpha m}+\frac{\kappa}{2}+\frac{\nu_0}{2}}\left[(\widetilde{A_{\infty}})^2 \right] = \mathcal{E}_{\alpha}(U_{\alpha}\mu).
\end{eqnarray*}

\end{proof}

In order to prove Theorem \ref{main2}, we prove the following type of contractions. Without loss of generality, we may assume that \(\alpha=1\).
\begin{lemma}\label{symSmooth}
For \(\mu_n, \mu \in \mathcal{S}\), its corresponding PCAF \(A^n, A \in {\bf A}_c^+\) and any \(t\geq 0\), it holds that
\[\int_E P_t \mathbb{E}_{\cdot}[(\widetilde{A^n_{\infty}}-\widetilde{A_{\infty}})^2](x) d\kappa \leq 2 \left( \sqrt{\mathcal{E}_1(U_1\mu_n)} +\sqrt{\mathcal{E}_1(U_1\mu)} \right) \sqrt{\mathcal{E}_1(U_1(\mu_n-\mu))} \]
and
\[\int_E P_t \mathbb{E}_{\cdot}[(\widetilde{A^n_{\infty}}-\widetilde{A_{\infty}})^2](x) d\nu_0 \leq 2 \left( \sqrt{\mathcal{E}_1(U_1\mu_n)} +\sqrt{\mathcal{E}_1(U_1\mu)} \right) \sqrt{\mathcal{E}_1(U_1(\mu_n-\mu))} \]
\end{lemma}
\begin{proof}
By taking nests \(\{F_k^{(n)}\}_k\) satisfying \({\bf 1}_{F_k^{(n)}} \mu_n \in \mathcal{S}_0\) and the same as the proof of Lemma \ref{nu_0Lemma}, we may assume that \(U_1\mu_n\) is bounded for each \(n\). Recall that \(2U_2((U_1\mu_n-U_1\mu)(\mu_n-\mu))\) is a quasi-continuous version of \(\mathbb{E}_{x}[(\widetilde{A^n_{\infty}}-\widetilde{A_{\infty}})^2]\) by \cite[Lemma 5.1.3]{FOT}. By the \(m\)-symmetry for \(P_t\), we have 
\begin{eqnarray*}
\int_E P_t \mathbb{E}_{\cdot}[(\widetilde{A^n_{\infty}}-\widetilde{A_{\infty}})^2](x) d\kappa &=& \mathcal{E}_2(U_2\kappa, 2 P_t U_2(U_1(\mu_n-\mu)(\mu_n-\mu)))\\
&=& 2\  \mathcal{E}_2(P_t U_2\kappa, U_2(U_1(\mu_n-\mu)(\mu_n-\mu)))\\
&= & 2 \int P_tU_2\kappa\ U_1(\mu_n-\mu)\  d(\mu_n-\mu)\\
&\leq &  2 \int P_tU_2\kappa\ |U_1(\mu_n-\mu)|\  d(\mu_n + \mu)\\
&\leq &  2 \int |U_1(\mu_n-\mu)|\  d(\mu_n + \mu)\\
&= & 2\ \mathcal{E}_1 (U_1\mu_n+U_1\mu, \left|U_1(\mu_n-\mu)\right|)\\
&\leq & 2 \left( \sqrt{\mathcal{E}_1(U_1\mu_n)} +\sqrt{\mathcal{E}_1(U_1\mu)} \right) \sqrt{\mathcal{E}_1(|U_1(\mu_n-\mu)|)} \\
&\leq & 2 \left( \sqrt{\mathcal{E}_1(U_1\mu_n)} +\sqrt{\mathcal{E}_1(U_1\mu)} \right) \sqrt{\mathcal{E}_1(U_1(\mu_n-\mu))}.
\end{eqnarray*}

Similarly, by using \cite[Theorem 5.4.3.(iv)]{CF} and \(|P_t\varphi_2 |\leq 1\), we have
\begin{eqnarray*}
\int_E P_t \mathbb{E}_{\cdot}[(\widetilde{A^n_{\infty}}-\widetilde{A_{\infty}})^2](x) d\nu_0 &=& \varlimsup_{s\searrow 0} \frac{1}{s} \int_E \mathbb{E}_x[(\widetilde{A^n_{\infty}}-\widetilde{A_{\infty}})^2] (1-e^{-2s}P_s)P_t\varphi_2(x) dm(x)\\
&=& \int_E 2U_1(\mu_n-\mu) P_t\varphi_2(x) d(\mu_n-d\mu)\\
&\leq &  2 \int |U_1(\mu_n-\mu)|\  d(\mu_n + \mu)\\
&\leq & 2 \left( \sqrt{\mathcal{E}_1(U_1\mu_n)} +\sqrt{\mathcal{E}_1(U_1\mu)} \right) \sqrt{\mathcal{E}_1(U_1(\mu_n-\mu))}.
\end{eqnarray*}
\end{proof}

We prepare the following lemmas concerning equivalent conditions for convergence of PCAFs.

\begin{lemma}\label{EnergyLem}
For \(A^n, A \in {\bf A}_c^+\), the following are equivalent.\\
{\rm (i)} There exists \(T>0\) such that \(\lim_{n\to \infty}\mathbb{E}_m\left[|\widetilde{A^n_T}-\widetilde{A_T}|^2\right]=0 \).\\
{\rm (ii)} For any \(T\geq 0\), \(\lim_{n\to \infty}\mathbb{E}_m\left[|\widetilde{A^n_T}-\widetilde{A_T}|^2\right]=0 \).\\
{\rm (iii)} There exists \(T>0\) such that \(\lim_{n\to \infty}\sup_{t\leq T}\mathbb{E}_m\left[|A^n_t-A_t|^2\right]=0 \).\\
{\rm (iv)} For any \(T\geq 0\), \(\lim_{n\to \infty}\sup_{t\leq T}\mathbb{E}_m\left[|A^n_t-A_t|^2\right]=0 \).\\
{\rm (v)} \(\lim_{n\to \infty}\mathbb{E}_m\left[|\widetilde{A^n_{\infty}}-\widetilde{A_{\infty}}|^2\right]=0.\)
\end{lemma}

\begin{proof}
It is clear that (ii) (resp. (iv)) implies (i) (resp. (iii)).

At first, we prove (v) implies (ii). We assume (v). Since \(\widetilde{A^n_{\infty}}=\widetilde{A^n_T}+e^{-T}\widetilde{A^n_{\infty}}\circ \theta_T\) and the \(m\)-symmetry of \(X\), we have
\begin{eqnarray}
\nonumber \mathbb{E}_m\left[|\widetilde{A^n_T}-\widetilde{A_T}|^2\right]& \leq &2\mathbb{E}_m\left[|\widetilde{A^n_{\infty}}-\widetilde{A_{\infty}}|^2 \right]+2e^{-2T}\mathbb{E}_m\left[|(\widetilde{A^n_{\infty}}-\widetilde{A_{\infty}})\circ \theta_T| ^2\right]\\
\nonumber &= &2\mathbb{E}_m\left[|\widetilde{A^n_{\infty}}-\widetilde{A_{\infty}}|^2 \right]+2e^{-2T}\mathbb{E}_m \mathbb{E}_{X_T}\left[|\widetilde{A^n_{\infty}}-\widetilde{A_{\infty}}|^2\right]\\
&\leq &(2+2e^{-2T})\mathbb{E}_m\left[|\widetilde{A^n_{\infty}}-\widetilde{A_{\infty}}|^2 \right]. \label{eq:energyT}
\end{eqnarray}
So (ii) holds.

Next, we prove (i) implies (v). We assume (i). For \(T>0\) satisfying the condition (i), we set \(\varepsilon:=e^T>1\) and \(C_{\varepsilon}:=2+\frac{1}{\varepsilon-1}\). Since we have \((a+b)^2\leq C_{\varepsilon} a^2 +\varepsilon b^2\) for any \(a,b \in \mathbb{R}\), it holds that
\begin{eqnarray*}
\mathbb{E}_m\left[|\widetilde{A^n_{\infty}}-\widetilde{A_{\infty}}|^2\right]&=&\mathbb{E}_m\left[\left|(\widetilde{A^n_{T}}-\widetilde{A_{T}})+e^{-T}(\widetilde{A^n_{\infty}}-\widetilde{A_{\infty}})\circ \theta_T \right| ^2\right]\\
&\leq & C_{\varepsilon} \mathbb{E}_m\left[|\widetilde{A^n_{T}}-\widetilde{A_{T}}|^2 \right]+\varepsilon e^{-2T} \mathbb{E}_m \mathbb{E}_{X_T}\left[|\widetilde{A^n_{\infty}}-\widetilde{A_{\infty}}|^2\right]\\
&\leq & C_{\varepsilon} \mathbb{E}_m\left[|\widetilde{A^n_{T}}-\widetilde{A_{T}}|^2 \right]+\varepsilon e^{-2T} \mathbb{E}_m\left[|\widetilde{A^n_{\infty}}-\widetilde{A_{\infty}}|^2\right],\\
\end{eqnarray*}
and we get
\begin{equation}\label{eq:EnegyLem1}
\mathbb{E}_m\left[|\widetilde{A^n_{\infty}}-\widetilde{A_{\infty}}|^2\right] \leq \frac{C_{\varepsilon}}{1-\varepsilon e^{-2T}} \mathbb{E}_m\left[|\widetilde{A^n_{T}}-\widetilde{A_{T}}|^2 \right]
\end{equation}
Here we remark that \(1-\varepsilon e^{-2T}=1-e^{-T}>0.\) So (v) holds.

We prove (iii) implies (i). We take \(T>0\) satisfying (iii), then, by using an integration by parts formula for Lebesgue-Stieltjes integrations, we have
\begin{equation}
\widetilde{A^n_T}-\widetilde{A_T} = e^{-T}(A_T^n-A_T) +\int_0^T e^{-t}(A_t^n-A_t) dt \label{eq:bubunsekibun}
\end{equation}
and so we have
\begin{eqnarray*}
\mathbb{E}_m\left[(\widetilde{A^n_{T}}-\widetilde{A_{T}})^2 \right] &\leq & 2e^{-2T}\mathbb{E}_m\left[(A_T^n-A_T)^2 \right]+ 2\mathbb{E}_m\left[\left(\int_0^T e^{-t}(A_t^n-A_t) dt \right)^2 \right]\\
 &\leq & 2e^{-2T}\mathbb{E}_m\left[(A_T^n-A_T)^2 \right]+ 2T\int_0^T e^{-2t}\mathbb{E}_m\left[(A_t^n-A_t)^2  \right]dt \\
&\leq & (2e^{-2T}+T(1-e^{-T})) \sup_{0\leq t\leq T} \mathbb{E}_m\left[(A_t^n-A_t)^2 \right]
\end{eqnarray*}
So (i) holds.

Next, we prove (v) implies (iv). We assume (v). By (\ref{eq:bubunsekibun}), we have
\[e^{-t}|A_t^n-A_t| \leq |\widetilde{A^n_t}-\widetilde{A_t}|+\int_0^t e^{-s}|A_s^n-A_s| ds.\]
By Gr\"{o}nwall's inequality, we have
\[e^{-t}|A_t^n-A_t| \leq |\widetilde{A^n_t}-\widetilde{A_t}|+\int_0^t |\widetilde{A^n_s}-\widetilde{A_s}| e^{t-s} ds.\]
Thus, by H\"older's inequality and (\ref{eq:energyT}), we have
\begin{eqnarray*}
\mathbb{E}_m\left[|A_t^n-A_t|^2 \right] &\leq & 2e^{2t} \mathbb{E}_m\left[|\widetilde{A^n_t}-\widetilde{A_t}|^2\right] 
+ 2\mathbb{E}_m\left[\left(\int_0^t |\widetilde{A^n_s}-\widetilde{A_s}| e^{2t-s}ds\right)^2 \right]\\
&\leq &  2e^{2t}\mathbb{E}_m\left[ |\widetilde{A^n_t}-\widetilde{A_t}|^2 \right] + 2\int_0^t \mathbb{E}_m\left[|\widetilde{A^n_s}-\widetilde{A_s}|^2\right] ds \int_0^t e^{4t-2s}ds \\
&\leq &  2e^{2t}(2+2e^{-2t})\mathbb{E}_m\left[|\widetilde{A^n_{\infty}}-\widetilde{A_{\infty}}|^2 \right] + 2\int_0^t (2+2e^{-2s})\mathbb{E}_m\left[|\widetilde{A^n_{\infty}}-\widetilde{A_{\infty}}|^2 \right] ds \frac{e^{4t}-e^{2t}}{2}\\
&= &  \left(4(e^{2t}+1) +(e^{4t}-e^{2t})(2t+1-e^{-2t})\right) \mathbb{E}_m\left[|\widetilde{A^n_{\infty}}-\widetilde{A_{\infty}}|^2 \right].
\end{eqnarray*}
By taking the supremum, (iv) holds.

\end{proof}

The following lemma can be proved in similar ways to the proofs of Theorem \ref{submain1} and Lemma \ref{EnergyLem}. Instead of the symmetry and the contraction of \(P_t\) with respect to \(m\), we use Lemma \ref{symSmooth}. Since the difference is only due to technical reasons, the proof of the following lemma is provided in \ref{sec app}.

\begin{lemma}\label{EnergykappaLem}
For \(A^n, A \in {\bf A}_c^+\), the following are equivalent.\\
{\rm (i)} For any \(T\geq 0\), \(\lim_{n\to \infty}\mathbb{E}_{m+\kappa+\nu_0}\left[|\widetilde{A^n_T}-\widetilde{A_T}|^2\right]=0 \).\\
{\rm (ii)} For any \(T\geq 0\), \(\lim_{n\to \infty}\sup_{t\leq T}\mathbb{E}_{m+\kappa+\nu_0}\left[|A^n_t-A_t|^2\right]=0 \).\\
{\rm (iii)} \(\lim_{n\to \infty}\mathbb{E}_{m+\kappa+\nu_0}\left[|\widetilde{A^n_{\infty}}-\widetilde{A_{\infty}}|^2\right]=0.\)
\end{lemma}

\begin{remark}
If \(\kappa\) is an excessive measure, that is, \(\int P_t f d\kappa \leq \int f d\kappa\) for any positive Borel function, the equivalence in Lemma \ref{EnergyLem}  with \(m\) replaced by \(\kappa\) holds.
\end{remark}

\begin{proof}[Proof of Theorem \ref{main2}]
We remark that \(\mathbb{E}_{m+\kappa /2+\nu_0/2}\) is comparable to \(\mathbb{E}_{m+\kappa+\nu_0}\). 
We assume 
\[\lim_{n\to \infty}\mathbb{E}_{m+\kappa+\nu_0}\left[\sup_{0\leq t\leq T}|A^n_t-A_t|^2\right]=0\] for any \(T>0\). Then \(\sup_{0\leq t\leq T}\mathbb{E}_{m+\kappa+\nu_0}\left[|A^n_t-A_t|^2\right]\) also converges to \(0\). By Lemma \ref{EnergykappaLem}, it holds that
\begin{equation}\label{eq:main2-1}
\lim_{n\to \infty}\mathbb{E}_{m+\kappa+\nu_0}\left[|\widetilde{A^n_{\infty}}-\widetilde{A_{\infty}}|^2\right]=0.
\end{equation}
Hence, by (\ref{eq:main2-A}), \(\mu_n\) converges to \(\mu\) in \(\rho\).

Next, we assume that \(\mu_n\) converges to \(\mu\) in \(\rho.\) Then, by Lemma \ref{EnergykappaLem}, we have, for any \(T\geq 0\), \[\lim_{n\to \infty}\sup_{0\leq t\leq T}\mathbb{E}_{m+\kappa+\nu_0}\left[|A_t^n-A_t|^2\right]=0.\]

For fixed \(T>0\), we set \(N:=2^k\) and \(t_i:=iT/2^k\) for integers \(0\leq i \leq N.\) Then we have, for any \(t\leq T\),
\begin{eqnarray*}
|A^n_t-A_t|&\leq & \max_{1\leq i\leq N} |A^n_{t_i}-A_{t_i}|+ \sum_{i=0}^{N-1}  |A^n_{t}-A^n_{t_i}|{\bf 1}_{(t_i, t_{i+1}]}(t) + \sum_{i=0}^{N-1}  |A_{t}-A_{t_i}|{\bf 1}_{(t_i, t_{i+1}]}(t)\\
&\leq & \max_{1\leq i\leq N} |A^n_{t_i}-A_{t_i}|+ \sum_{i=0}^{N-1}  |A^n_{t_{i+1}}-A^n_{t_i}|{\bf 1}_{(t_i, t_{i+1}]}(t) + \sum_{i=0}^{N-1}  |A_{t_{i+1}}-A_{t_i}|{\bf 1}_{(t_i, t_{i+1}]}(t)\\
&\leq & \max_{1\leq i\leq N} |A^n_{t_i}-A_{t_i}|+ \sum_{i=0}^{N-1}  |A^n_{\frac{T}{N}}\circ \theta_{t_i}|{\bf 1}_{(t_i, t_{i+1}]}(t) + \sum_{i=0}^{N-1} |A_{\frac{T}{N}}\circ \theta_{t_i}|{\bf 1}_{(t_i, t_{i+1}]}(t),
\end{eqnarray*}
and we have
\begin{eqnarray}
\nonumber \mathbb{E}_m[\sup_{t\leq T}|A^n_t-A_t|^2]&\leq &3 \sum_{i=1}^N\mathbb{E}_m[|A^n_{t_i}-A_{t_i}|^2] +3\sum_{i=0}^{N-1}  \mathbb{E}_m[|A^n_{\frac{T}{N}}\circ \theta_{t_i}|^2]  + 3\sum_{i=0}^{N-1} \mathbb{E}_m[|A_{\frac{T}{N}}\circ \theta_{t_i}|^2] \\
&\leq & 3N \sup_{t\leq T}\mathbb{E}_m[|A^n_{t}-A_{t}|^2]+3N \mathbb{E}_m[|A^n_{\frac{T}{N}}|^2] + 3N \mathbb{E}_m[|A_{\frac{T}{N}}|^2]. \label{eq:localU-1}
\end{eqnarray}

By \cite[p.245]{FOT}, we have
\begin{eqnarray}
\nonumber \lefteqn{N \mathbb{E}_m[|A^n_{\frac{T}{N}}|^2] }\\
\nonumber &\leq & 2Ne^{\frac{T}{N}}\mathbb{E}_m\left[\int_0^{\frac{T}{N}} e^{-s}\mathbb{E}_{X_s}[A_{\frac{T}{N}}^n]dA_s^n  \right]\\
&\leq & 2T e^{\frac{2T}{N}} \mathcal{E}_1\left(U_1\mu_n, (1-e^{-\frac{T}{N}}P_{\frac{T}{N}})U_1\mu_n \right) \label{eq:localU-1-2}\\
\nonumber &\leq & 2T e^{\frac{2T}{N}} \left|\mathcal{E}_1\left(U_1\mu_n -U_1\mu , (1-e^{-\frac{T}{N}}P_{\frac{T}{N}})U_1\mu_n \right) \right|\\
\nonumber && \hspace{20mm}+ 2T e^{\frac{2T}{N}} \left|\mathcal{E}_1\left(U_1\mu , (1-e^{-\frac{T}{N}}P_{\frac{T}{N}})(U_1\mu -U_1\mu_n) \right) \right|+ 2T e^{\frac{2T}{N}}\mathcal{E}_1\left(U_1\mu , (1-e^{-\frac{T}{N}}P_{\frac{T}{N}})U_1\mu  \right)\\
\nonumber &\leq & 2T e^{\frac{2T}{N}} \sqrt{\mathcal{E}_1\left(U_1\mu_n -U_1\mu \right)} \sqrt{\mathcal{E}_1\left((1-e^{-\frac{T}{N}}P_{\frac{T}{N}})U_1\mu_n \right)}\\
 && + 2T e^{\frac{2T}{N}} \sqrt{\mathcal{E}_1\left((1-e^{-\frac{T}{N}}P_{\frac{T}{N}})U_1\mu \right)} \sqrt{\mathcal{E}_1(U_1\mu -U_1\mu_n)}+ 2T e^{\frac{2T}{N}}\mathcal{E}_1\left(U_1\mu , (1-e^{-\frac{T}{N}}P_{\frac{T}{N}})U_1\mu  \right). \label{eq:localU-2}
\end{eqnarray}
Since \(\mathcal{E}_1(e^{-t}P_tf) \leq \mathcal{E}_1(f)\) holds for \(f\in \mathcal{F}\), we have
\begin{equation}
\mathcal{E}_1\left((1-e^{-\frac{T}{N}}P_{\frac{T}{N}})U_1\mu_n \right) \leq 4 \mathcal{E}_1(U_1\mu_n) \label{eq:localU-3}
\end{equation}
and
\begin{equation}
\mathcal{E}_1\left((1-e^{-\frac{T}{N}}P_{\frac{T}{N}})U_1\mu  \right) \leq 4 \mathcal{E}_1(U_1\mu). \label{eq:localU-4}
\end{equation}
Hence, by (\ref{eq:localU-1}), (\ref{eq:localU-2}), (\ref{eq:localU-3}),  (\ref{eq:localU-4}) and \(\sup_n \rho(\mu_n, \mu_n)<\infty \), we have
\begin{eqnarray}
\nonumber \mathbb{E}_m[ \sup_{t\leq T}|A^n_t-A_t|^2] &\leq & 3N \sup_{t\leq T}\mathbb{E}_m[|A^n_{t}-A_{t}|^2]+C\sqrt{\mathcal{E}_1(U_1\mu -U_1\mu_n)}\\
&& \hspace{20mm}+C\mathcal{E}_1\left(U_1\mu , (1-e^{-\frac{T}{N}}P_{\frac{T}{N}})U_1\mu  \right) \label{eq:localU-5}
\end{eqnarray}
for \(C>0\) independent of \(n,N\). By \cite[p.245]{FOT}, for any \(\varepsilon>0,\) we can take large \(N\) satisfying 
\[\mathcal{E}_1\left(U_1\mu , (1-e^{-\frac{T}{N}}P_{\frac{T}{N}})U_1\mu  \right) <\varepsilon, \] and, by letting \(n\) tends to infinity in (\ref{eq:localU-5}), we have
\[\varlimsup_{n\to \infty} \mathbb{E}_m[ \sup_{t\leq T}|A^n_t-A_t|^2]\leq C\varepsilon. \]
Similarly, we have
\begin{eqnarray}
\mathbb{E}_{\kappa}[\sup_{t\leq T}|A^n_t-A_t|^2]\leq  3N \sup_{t\leq T}\mathbb{E}_{\kappa}[|A^n_{t}-A_{t}|^2]+3\sum_{i=0}^{N-1} \int P_{t_i} \mathbb{E}_{\cdot}[(A^n_{\frac{T}{N}})^2]d\kappa  + 3\sum_{i=0}^{N-1} \int P_{t_i} \mathbb{E}_{\cdot}[(A_{\frac{T}{N}})^2]d\kappa \label{eq:localU-6}
\end{eqnarray}
and, using a nest \(\{F_l\}\) satisfying \({\bf 1}_{F_l}\kappa \in \mathcal{S}_0\) and (\ref{eq:Ukappa}), we have

\begin{eqnarray}
\nonumber \sum_{i=0}^{N-1} \int P_{t_i} \mathbb{E}_{\cdot}[(A^n_{\frac{T}{N}})^2]d\kappa &\leq &2\sum_{i=0}^{N-1} \int P_{t_i} \mathbb{E}_{\cdot}\left[\int_0^{\frac{T}{N}}\mathbb{E}_{X_s}[A_{\frac{T}{N}}^n]dA_s^n  \right]  d\kappa\\
\nonumber &\leq & 2 \mathbb{E}_{\kappa} \left[\int_0^{T}\mathbb{E}_{X_s}[A_{\frac{T}{N}}^n]dA_s^n  \right]\\
\nonumber &\leq & 2e^T \int _E \left(U_{A^n}^1\mathbb{E}_{\cdot}[A_{\frac{T}{N}}^n]\right)(x) d\kappa(x)\\
\nonumber &=&2e^T \lim_{l\to \infty} \mathcal{E}_1 \left( U_{A^n}^1\mathbb{E}_{\cdot}[A_{\frac{T}{N}}^n], U_1({\bf 1}_{F_l}\kappa )\right)\\
\nonumber &=& 2e^T  \lim_{l\to \infty} \int _E U_1({\bf 1}_{F_l}\kappa ) \mathbb{E}_x[A_{\frac{T}{N}}^n] d\mu_n(x)\\
\nonumber &\leq & 2e^T  \int _E  \mathbb{E}_x[A_{\frac{T}{N}}^n] d\mu_n(x)\\
&\leq &2e^Te^{\frac{T}{N}} \mathcal{E}_1\left(U_1\mu_n, (1-e^{-\frac{T}{N}}P_{\frac{T}{N}})U_1\mu_n \right)\label{eq:localU-8}
\end{eqnarray}

Here, in (\ref{eq:localU-8}), we used the same way as \cite[p.245]{FOT}. Since the last term in (\ref{eq:localU-8}) is the same one as (\ref{eq:localU-1-2}) up to constant, for any \(\varepsilon>0,\) we have
\[\varlimsup_{n\to \infty} \mathbb{E}_{\kappa}[ \sup_{t\leq T}|A^n_t-A_t|^2]\leq C\varepsilon. \]

Since \(\varphi_2(x):=\mathbb{E}_x[e^{-2\zeta}{\bf 1}_{\{\partial\}}(x_{\zeta -})]\) plays the same role as \(U_2 \kappa\), by the same way as above and Lemma \ref{nu_0Lemma}, we have
\[\varlimsup_{n\to \infty} \mathbb{E}_{\nu_0}[ \sup_{t\leq T}|A^n_t-A_t|^2]\leq C\varepsilon. \]

So the proof is completed.
\end{proof}

\begin{corollary}\label{main2a}
Suppose that \(X\) is conservative, that is, \(\mathbb{P}_x(\zeta=\infty)=1\) for any \(x.\) For \(\mu _n, \mu \in \mathcal{S}_0\), let \(A^n, A\) be the corresponding PCAFs respectively. Revuz measures \(\mu_n\) converges to \(\mu\) in \(\rho\) if and only if there exists \(T>0\) such that 
\[\lim_{n\to \infty}\mathbb{E}_{m}\left[\sup_{0\leq t\leq T}|A^n_t-A_t|^2\right]=0.\]
In addition, this convergence is also equivalent to the existence of \(T>0\) such that 
\[\lim_{n\to \infty} \sup_{t\leq T} \mathbb{E}_{m}\left[|A^n_t-A_t|^2\right]=0.\]

Moreover, for \(\mu, \nu \in \mathcal{S}_0\), let \(A, B\) be the corresponding PCAFs respectively, then it holds that 
\begin{equation*}
\mathbb{E}_{\alpha m}\left[\widetilde{A_{\infty}}\widetilde{B_{\infty}}\right] =\mathcal{E}_{\alpha}(U_{\alpha}\mu, U_{\alpha}\nu)
\end{equation*}
where we set \(\widetilde{A_t}:=\int_0^t e^{-\alpha s}dA_s\) and \(\widetilde{B_t}:=\int_0^t e^{-\alpha s}dB_s\) for \(\alpha>0\).
\end{corollary}

\begin{proof}
This follows from Theorem \ref{main2}, Proposition \ref{main2-2} and Lemma \ref{EnergyLem}.
\end{proof}

\begin{remark}
Let \({\bf A}_{c,0}^+:= {\bf A}_{c}^+ \cap L^2(\mathbb{P}_{m+\kappa+\nu_0})\). Then, we can paraphrase Theorem \ref{main2} as follows. The Revuz map from \(\mathcal{S}_0\) equipped with the topology \(\rho\), to \({\bf A}_{c,0}^+\) equipped with the topology induced by \(L^2(\mathbb{P}_{m+\kappa +\nu_0})\) with the local uniform topology, is a homeomorphism.
\end{remark}

\section{Vague convergence of Revuz measures}\label{mainsec2}
In this section, we consider the conditions under which vague convergence for Revuz measures belonging to \(\mathcal{S}_0\) follows from convergence of PCAFs. Let \((\mathcal{E}, \mathcal{F})\) be a regular Dirichlet form on \(L^2(E;m)\) associated with an \(m\)-symmetric Hunt process \(X\), \(\kappa\) be the killing measure of \(X\) and \(\nu_0\) be the energy functional defined in Section \ref{mainsec}. We make the following assumption.
\begin{assumption}\label{ass1}
One of the following conditions holds\\
(1) There exists \(T>0\) such that \(\sup_n \mathbb{E}_{m+\kappa +\nu_0}[(A_T^n)^2]<\infty \)\\
(2) There exists \(T>0\) such that \(\sup_n \mathbb{E}_m[(A_T^n)^2]<\infty \), and, for any compact set \(K\), there exists \(0\leq C_K<1\) such that \(\mathbb{E}_x[e^{-\zeta}]\leq C_K\) for q.e. \(x\in K.\)
\end{assumption}
Assumption \ref{ass1} (2) is intended to prevent immediate killing inside of a process.

The following is one of the main results, concerning a necessary condition for convergence of Revuz measures under the vague topology.
\begin{theorem}\label{mainvague}
For \(\mu _n, \mu \in \mathcal{S}_0\), let \(A^n, A\in {\bf A}_c^+\) be the corresponding PCAFs respectively. Suppose that Assumption \ref{ass1} holds. If \(A^n\) converges to \(A\) \(\mathbb{P}_x\)-almost surely for q.e. \(x\in E\) with the local uniform topology, then \(\mu_n\) converges to \(\mu\) vaguely.
\end{theorem}

Before beginning the proof of Theorem \ref{mainvague}, we state the following lemmas concerning the uniform boundedness of PCAFs. In contrast to Lemma \ref{EnergykappaLem}, which was proved using Proposition \ref{symSmooth}, these lemmas are proved in a similar way to the proof of Lemma \ref{EnergyLem} by using \[\int P_t \mathbb{E}_{\cdot}[(\widetilde{A_{\infty}^n})^2] d(m+\frac{\kappa}{2}+\frac{\nu_0}{2}) = \int P_t1\  U_1\mu_n d\mu_n \leq \int  U_1\mu_n d\mu_n  = \mathbb{E}_{m+\frac{\kappa}{2}+\frac{\nu_0}{2}}[(\widetilde{A_{\infty}^n})^2] .\]
\begin{lemma}\label{UIlem}
The following are equivalent.\\
{\rm (i)} There exists \(T>0\) such that \(\sup_n \mathbb{E}_m[(A_T^n)^2]<\infty \).\\
{\rm (ii)} There exists \(T>0\) such that \(\sup_n \mathbb{E}_m[(\widetilde{A_T^n})^2]<\infty \).\\
{\rm (iii)} For any \(T\geq 0\), \(\sup_n \mathbb{E}_m[(A_T^n)^2]<\infty \).\\
{\rm (iv)} For any \(T\geq 0\), \(\sup_n \mathbb{E}_m[(\widetilde{A_T^n})^2]<\infty \).\\
{\rm (v)} \(\sup_n \mathbb{E}_m[(\widetilde{A_{\infty}^n})^2]<\infty \).
\end{lemma}

\begin{lemma}\label{UIlemkappa}
The following are equivalent.\\
{\rm (i)} There exists \(T>0\) such that \(\sup_n \mathbb{E}_{m+\kappa +\nu_0}[(A_T^n)^2]<\infty \).\\
{\rm (ii)} There exists \(T>0\) such that \(\sup_n \mathbb{E}_{m+\kappa +\nu_0}[(\widetilde{A_T^n})^2]<\infty \).\\
{\rm (iii)} For any \(T\geq 0\), \(\sup_n \mathbb{E}_{m+\kappa +\nu_0}[(A_T^n)^2]<\infty \).\\
{\rm (iv)} For any \(T\geq 0\), \(\sup_n \mathbb{E}_{m+\kappa +\nu_0}[(\widetilde{A_T^n})^2]<\infty \).\\
{\rm (v)} \(\sup_n \mathbb{E}_{m+\kappa +\nu_0}[(\widetilde{A_{\infty}^n})^2]<\infty \).
\end{lemma}

\begin{proof}[Proof of Theorem \ref{mainvague}]
We fix \(f\in C_c(E)\) and take a compact set \(K:=\) supp\((f)\).

We set \((fA^n)_t:=\int_0^t f(X_s)dA^n_s\) then, by \cite[Lemma 5.1.3]{FOT}, \(fd\mu_n \in \mathcal{S}_0\) and \(U_{(fA^n)}^{\alpha}1=U_{A^n}^{\alpha}f\) is a quasi-continuous version of \(U_{\alpha}(f\mu_n)\). 

For any bounded function \(h\in L^1(m)\) and \(t>0\), we have
\begin{eqnarray*}
\mathbb{E}_{hm} \left[\left(\int_0^t f(X_s)dA_s^n \right)^2\right]&\leq & e^{2t} \mathbb{E}_{hm} \left[(\widetilde{(fA^n)}_{t})^2\right]\\
&\leq & e^{2t} ||h||_{\infty} ||f||_{\infty}^2 \mathbb{E}_{m}[(\widetilde{A^n_{\infty}})^2],
\end{eqnarray*}
and so, by Lemma \ref{UIlem}, \(\{\int_0^t f(X_s)dA_s^n \}_n\) is uniformly integrable under \(\mathbb{E}_{hm}\). Since we are now assuming that \(A^n\) converges to \(A\), \(\mathbb{P}_x\)-almost surely for q.e. \(x\in E\) with the local uniform topology, we fix q.e. \(x\), and take \(\Omega\) satisfying \(\mathbb{P}_x(\Omega)=1\) and \(\sup_{s\leq t}|A^n_s(\omega)-A_s(\omega)|\) converges to \(0\) for \(\omega \in \Omega\). By \cite[Lemma 4.8]{O}, for any \(\omega \in \Omega\), \(\int_0^t f(X_s)dA_s^n(\omega)\) converges to \(\int_0^t f(X_s)dA_s (\omega)\), so \(\int_0^t f(X_s)dA_s^n\) converges to \(\int_0^t f(X_s)dA_s\), \(\mathbb{P}_{hm}\)-almost everywhere. By the Vitali convergence theorem, we have
\begin{eqnarray*}
\lim_{n\to \infty}\mathbb{E}_{hm} \left[\int_0^t f(X_s)dA_s^n \right]=\mathbb{E}_{hm} \left[\int_0^t f(X_s)dA_s \right].
\end{eqnarray*}
By the correspondence between Revuz measures and PCAFs, it holds that
\[\mathbb{E}_{hm} \left[\int_0^t f(X_s)dA_s^n \right]=\int_E S_th \cdot fd\mu_n\]
and
\[\mathbb{E}_{hm} \left[\int_0^t f(X_s)dA_s \right]=\int_E S_th \cdot fd\mu,\]
where \(S_th:=\int_0^t P_sh ds.\) Therefore we have
\begin{equation}\label{eq:main1-2}
\lim_{n\to \infty}\int_E \frac{1}{t}S_th \cdot fd\mu_n=\int_E \frac{1}{t}S_th \cdot fd\mu.
\end{equation}

By \cite[Lemma 2.1.4]{CF}, \(S_th\in \mathcal{F}\) for any \(h\in L^2(m)\) and \(\mathcal{E}(S_th,u)=\langle h-P_th,u \rangle_m\) holds for any \(u\in \mathcal{F}\). Hence we have \[\mathcal{E}_1\left(\frac{1}{t}S_th-h\right)=\frac{1}{t^2}\langle h-P_th,S_th \rangle_m -\frac{2}{t}\langle h-P_th,h \rangle_m +\mathcal{E}(h)+\left \|\frac{1}{t}S_th-h \right \|_{L^2(m)}\]
and the right-hand side converges to \(0\) as \(t\) tends to \(0\).

Suppose Assumption \ref{ass1} (1), then, by Proposition \ref{main2-2} and Lemma \ref{UIlemkappa}, \(\{\mathcal{E}_1(U_1(f\mu_n))\}_n\) is bounded.

Suppose Assumption \ref{ass1} (2). In a similar manner to the proof of Proposition \ref{main2-2}, we have
\begin{eqnarray}
\mathbb{E}_m[(\widetilde{(fA^n)}_{\infty})^2]
= 2\int_K R_21 \cdot U_{1}(f\mu_n) \cdot f d\mu_n. \label{eq:mainvag-A}
\end{eqnarray}
By (\ref{eq:Ukappa}), we have
\begin{eqnarray*}
2R_21(x)=1-\mathbb{E}_x[e^{-2\zeta}]\geq 1-\mathbb{E}_x[e^{-\zeta}]>1-C_K>0
\end{eqnarray*}
for q.e. \(x\in K.\) Combining this with (\ref{eq:mainvag-A}), we have
\begin{eqnarray*}
\mathbb{E}_m[(\widetilde{(fA^n)}_{\infty})^2]&\geq &(1-C_K)\mathcal{E}_1(U_1(f\mu_n)),
\end{eqnarray*}
and so \(\{\mathcal{E}_1(U_1(f\mu_n))\}_n\) is bounded.

Under Assumption \ref{ass1}, for any bounded function \(h\in \mathcal{F}\cap L^1(m)\), we have
\begin{eqnarray*}
\lefteqn{\left|\int_E hf d\mu_n -\int_E hf d\mu \right|}\\
 &\leq & \left|\int_E \left(\frac{S_th}{t}-h\right)f d\mu_n \right| +
\left|\int_E \frac{S_th}{t}f d\mu_n -\int_E \frac{S_th}{t}f d\mu \right| +
\left|\int_E \left(\frac{S_th}{t}-h\right)f d\mu \right|\\
&\leq & \left(\sup_{n}\sqrt{\mathcal{E}_1(U_1(f\mu_n))}+\sqrt{\mathcal{E}_1(U_1(f\mu))}\right)\sqrt{\mathcal{E}_1\left(\frac{1}{t}S_th-h\right)}+\left|\int_E \frac{S_th}{t}f d\mu_n -\int_E \frac{S_th}{t}f d\mu \right|\\
\end{eqnarray*}
Therefore, letting \(n\) tend to infinity and then \(t\) tend to \(0\), by \((\ref{eq:main1-2})\), we get
\begin{equation}\label{eq:main1-3}
\lim_{n\to \infty}\int_E h \cdot fd\mu_n=\int_E h \cdot fd\mu
\end{equation}
for any bounded function \(h\in \mathcal{F}\cap L^1(m)\). Since \((\mathcal{E}, \mathcal{F})\) is regular, \((\ref{eq:main1-3})\) holds for any \(h\in \mathcal{F}\) similarly, and by taking \(h\in\mathcal{F}\cap C_c(E)\) satisfying \(h=1\) on \(K\), we have
\begin{equation}\label{eq:main1-4}
\lim_{n\to \infty}\int_E fd\mu_n=\lim_{n\to \infty}\int_E h \cdot fd\mu_n=\int_E h \cdot fd\mu=\int_E  fd\mu.
\end{equation}
So the proof is completed.
\end{proof}

\begin{corollary}\label{corvague}
For \(\mu _n, \mu \in \mathcal{S}_0\), let \(A^n, A\in {\bf A}_c^+\) be the corresponding PCAFs respectively. Suppose that Assumption \ref{ass1} (1) holds. If \(A^n\) converges to \(A\) \(\mathbb{P}_x\)-almost surely for q.e. \(x\in E\) with the local uniform topology, \(U_1\mu_n\) converges to \(U_1\mu\) weakly in \(\mathcal{E}_1.\)
\end{corollary}
\begin{proof}
This follows from Theorem \ref{mainvague}, Proposition \ref{main2-2} and the boundedness of \(\{ \mathcal{E}_1(U_1\mu_n)\}_n\).
\end{proof}

\section{Examples}\label{secexample}
\begin{example}\label{ex1}
The first example illustrates that the vague convergence of Revuz measures does not imply convergence of the corresponding PCAFs in general. In this case, the corresponding PCAF does not even converge in distribution.

Let \(\tau^{(k)}\) be \([0, \infty)\)-valued independent exponentially distributed random variables satisfying \(\mathbb{P} (\tau^{(k)} \leq t)=1-e^{-t}\) for any \(t\geq 0\), and define \(\tau^{(0)}:=0.\)\\
Let \(m\) be the Lebesgue measure on \(\mathbb{R}\) and we define the \(m\)-symmetric Hunt process \(X\) on \(\mathbb{R}\) by
\[X_t:=\left\{\begin{split} &\ \ x& : \sum_{k=0}^{2i}\tau^{(k)}\leq t< \sum_{k=0}^{2i+1}\tau^{(k)},\\ &-x& :\sum_{k=0}^{2i+1}\tau^{(k)}\leq t<\sum_{k=0}^{2i+2}\tau^{(k)},\end{split} \right.
\]
for \(t\geq 0\), \(x\in \mathbb{R}\) and \(i\in \mathbb{Z}_{\geq 0}.\) This is one of \textit{pure jump step processes}. See \cite[Section 2.2.1]{CF} for details.

We define the PCAF \(A^n\) by
\[A_t^n:=\int_0^t (\sin{(nX_s)}+1)ds.\]
The Revuz measure corresponding to \(A^n\) is \(\mu ^n(dx)=(\sin{(nx)}+1)dx.\) We define \(\mu\) to be the Lebesgue measure and \(A_t:=t\). Then, by the Riemann-Lebesgue lemma, \(\mu_n\) converges vaguely to \(\mu \), but \(A^n\) does not converge to \(A\) with respect to the local uniform topology in distribution as \(n\to \infty\) for any starting point \(x\). Moreover, this does not converge in distribution for any fixed time \(t\).
\begin{proof}
Fix \(t>0\). By the representation of \(X\), we have
\begin{eqnarray*}
A_t^n=\sum_{j=0}^{\infty}(1+(-1)^j\sin{(nx)})\left((t-\sum_{k=0}^{j}\tau^{(k)})_+\wedge \tau^{(j+1)}\right).
\end{eqnarray*}
By calculation, we have
\[
\varliminf_{n\to \infty} \mathbb{E}_x[A_t^n]=\varliminf_{n\to \infty} t(1+\sin{(nx)}\frac{\sinh t}{e^t})\ <\ t=\mathbb{E}_x[A_t],
\]
so \(A^n_t\) does not converge to \(A_t\) in distribution.
\end{proof}
Of independent interest, in a similar way to the above, we can also prove the sequence \(\{A^n\}_n\) under \(\mathbb{P}_x\) is not tight with respect to the local uniform topology for any \(x\). See \cite{JS, W} for details on the tightness and the relation between the tightness and convergence in distribution with respect to the local uniform topology.

We remark that \(\mu_n\) does not converge to \(\mu\) in \(\rho\) in this example, so we cannot apply Theorem \ref{main2}, \ref{submain1}.

\end{example}

\begin{example}\label{ex2}
Let \(X\) be Brownian motion on \(\mathbb{R}\). Then, the capacity of any one point is positive, so, for any \(x\in \mathbb{R}\), the Dirac measure \(\delta_x\) is a smooth measure and the corresponding PCAF \(L^x\) is a local time at \(x\in \mathbb{R}.\) In this case, \(\kappa =0\) and \(\nu_0=0\). We set the 1-order Green's function \(g_1\) by \[g_1(x,y)=\int_0^{\infty}e^{-t}\frac{1}{\sqrt{2\pi t}}e^{-\frac{|x-y|^2}{2t}}.\]
It follows from \cite[Exercise 4.2.2]{FOT} that \(\delta_x \in \mathcal{S}_{0}\), \(\mathcal{E}_1(U_1\delta_x, U_1\delta_y)=g_1(x,y)\) and \(U_1\delta_x(y)=g_1(x,y)\) for  any \(x, y\in \mathbb{R}.\) Moreover, \(\sup_{x\in \mathbb{R}}||U_1 \delta_x||_{\infty}=g_1(0,0) <\infty\) holds. For a sequence \(x_n \in \mathbb{R}\) converging to \(x\in \mathbb{R}\), by the continuity of the 1-order Green's function  of Brownian motion on \(\mathbb{R}\), \(\delta_{x_n}\) converges to \(\delta_x\) in \(\rho.\) By Corollary \ref{submain2}, we have
\[\lim_{n\to \infty}\mathbb{E}_y\left[ \sup_{t\leq T}|L^{x_n}_t-L^{x}_t|^2 \right]=0\]
for any \(y\in \mathbb{R}\) and \(T\geq 0\), and, by Theorem \ref{main2}, 
\[\lim_{n\to \infty}\mathbb{E}_m\left[ \sup_{t\leq T}|L^{x_n}_t-L^{x}_t|^2 \right]=0\]
for the Lebesgue measure \(m.\)
\end{example}

\begin{example}\label{ex3}
We give an example of Theorem \ref{main2} for absolutely continuous measures with respect to the underlying measure. Let \((\mathcal{E}, \mathcal{F})\) be a regular Dirichlet form on \(L^2(E;m)\) associated with a Hunt process \(X\). For \(f_n, f\in L^2(E;m)\), we assume that \(f_n\) converges to \(f\) in \(L^2(E;m)\), and we set \(d\mu_n:=f_ndm\) and \(d\mu:=fdm\). Then, their corresponding PCAFs are \(A_t^n:=\int_0^t f_n(X_s)ds\) and \(A_t:=\int_0^t f(X_s)ds\) respectively. It follows that
\[\rho(\mu_n, \mu)= \mathcal{E}_1(R_1f_n-R_1f)=\langle f_n-f, R_1(f_n-f)\rangle_m \leq \|f_n-f\|^2_m \]
and this converges to \(0\). By Theorem \ref{main2}, \(A^n\) converges to \(A\) in \(L^2(\mathbb{P}_{m+\kappa+\nu_0})\) with the local uniform topology.
\end{example}

\begin{example}\label{ex4}
In contrast to the previous example, we give an example of a singular measure. Let \(X\) be Brownian motion on \(\mathbb{R}\) and \(m\) be the Lebesgue measure on \(\mathbb{R}\). We construct the Cantor set according to \cite[pp.144-145]{Ru} as follows. Set \(C_1:=[0, \frac{1}{3}]\cup [\frac{2}{3},1]\) and, inductively, \(C_{n+1}:=\frac{1}{3}C_n \cup (\frac{2}{3}+\frac{1}{3}C_n)\) for positive integers \(n.\) Here we used the notation \(a+b[c,d]:=[a+bc, a+bd]\) for \(a,b,c,d\in \mathbb{R}\). Then \(C:=\cap_{n=1}^{\infty}C_n\) is called the Cantor set. We set \(d\mu_n:=\frac{1}{m(C_n)}m|_{C_n}\), then \(\mu_n\) converges weakly to \(\mu\) because their distribution functions converge uniformly by \cite[p.145]{Ru}. Since \(m(C)=0\) and \(\mu(C)=1\), \(\mu\) is singular with respect to \(m\). For the 1-order Green's function \(g_1\), it holds that
\[\int_E \int_E g_1(x,y) d\mu(x) d\mu(y) \leq g_1(0,0)<\infty, \]
so \(\mu \in \mathcal{S}_0\). Let \(A^n, A\) be PCAFs corresponding to \(\mu_n, \mu\) respectively. We remark that it holds that \(A^n_t:=\int_0^t \frac{1}{m(C_n)} {\bf 1}_{C_n}(X_s)ds,\) but we cannot obtain an explicit expression for \(A_t\). By \cite[p.145]{Ru}, \(\mu_n\) converges to \(\mu\) in \(\rho\). Hence, by Theorem \ref{main2}, \(A^n\) converges to \(A\) in \(L^2(\mathbb{P}_{m})\) with respect to the local uniform topology.
\end{example}

\begin{example}\label{ex4-2}
In the following simple example, we show that convergence with respect to the killing measure \(\kappa\) is needed for Theorem \ref{main2} in general. Let \(E\) be an open interval \((0,1)\) and \(m\) be the Lebesgue measure on \(E\). We set \(g(x):=x^{-2}\), \(\mathcal{F}:=L^2(E;dx) \cap L^2(E;gdx)\) and \(\mathcal{E}(f):=\int_E |f|^2g dx\). Then \((\mathcal{E}, \mathcal{F})\) is a regular Dirichlet form on \(L^2(E;dx)\) and the associated Hunt process \(X\) is a process that does not move until an exponentially distributed lifetime with rate \(g,\) that is, \(P_tu(x)=\mathbb{E}_x[\exp{(\int_0^t g(x)ds)}u(x)]=\exp{(g(x)t)}u(x)\). This Dirichlet form is called a perturbed Dirichlet form, and \(d\kappa = gdm\). See \cite[\S 5.1]{CF}.

We set  \(f_n(x):=(1+g(x))\left(1+\frac{1}{\sqrt{n}}\sin{(nx)}\right), f(x):=1+g(x), d\mu_n = f_ndx, d\mu :=fdx \in \mathcal{S}_0\) and \(A^n, A \in {\bf A}_c^+\) be the corresponding PCAFs respectively. Since \(R_1u=\frac{u}{1+g}\) for \(u\in L^2(E;dx)\), we have
\begin{eqnarray*}
\lim_{n\to \infty}\mathbb{E}_m\left[\sup_{t\leq T}|A_t^n-A_t|^2 \right] &\leq & \lim_{n\to \infty} e^{2T} \mathbb{E}_m\left[\left( \int_0^{\infty}e^{-t}|f_n-f|(X_t)dt \right)^2 \right]\\
&= & \lim_{n\to \infty} 2e^{2T} \int_0^1 R_21\ |f_n-f|\ R_1|f_n-f|\ dx\\
&= & \lim_{n\to \infty} 2e^{2T} \int_0^1 \frac{(\sin{(nx)})^2}{n}dx \\
&=& 0.
\end{eqnarray*}
However, we have
\begin{eqnarray*}
\lim_{n\to \infty} \rho(\mu_n, \mu) &=&\lim_{n\to \infty}  \int_0^1 (f_n-f)\ R_1(f_n-f)\ dx\\
&=&\lim_{n\to \infty} \int_0^1 (1+g(x))\frac{(\sin{(nx)})^2}{n}dx\\
&=& \lim_{n\to \infty} \int_0^1 \frac{(\sin{(nx)})^2}{n}dx + \lim_{n\to \infty} \int_0^n \frac{(\sin{(y)})^2}{y^2}dy\\
&=& \infty.
\end{eqnarray*}
In the third equation, we exchanged \(y=nx.\) Since the distribution of the lifetime \(\zeta\) under \(\mathbb{P}_x\) is an exponential distribution with rate \(g(x)\), we have
\begin{eqnarray*}
\lim_{n\to \infty}\mathbb{E}_{\kappa}\left[|A_t^n-A_t|^2 \right] 
&= & \lim_{n\to \infty} \int_0^1 \mathbb{E}_x[(t\wedge \zeta)^2] (1+g(x))^2  \frac{(\sin{(nx))^2}}{n} g(x)dx\\
&\geq & \lim_{n\to \infty} \int_0^1 \mathbb{E}_x[\zeta ^2 {\bf 1}_{\{\zeta \leq t\}}] (1+g(x))^2  \frac{(\sin{(nx))^2}}{n} g(x)dx\\
&\geq & \lim_{n\to \infty} \int_0^1 \frac{C_t}{g(x)^2} (1+g(x))^2  \frac{(\sin{(nx))^2}}{n} g(x)dx\\
&=& \infty.
\end{eqnarray*}
Hence, \(A^n\) converges to \(A\) in \(L^2(\mathbb{P}_m)\) but neither \(\mu_n\) nor \(A^n\) converge to \(\mu\) in \(\rho\) and \(A\) in \(L^2(\mathbb{P}_{\kappa})\) respectively.
\end{example}

\begin{example}\label{ex4-3}
In the following example, we show that convergence with respect to \(\nu_0\) is needed for Theorem \ref{main2} in general.

Let \(E\) be an open interval \((0,\infty)\), \(m\) be the Lebesgue measure on \(E\) and \(X\) be an absorbed Brownian motion on \(E.\) Then the \(\alpha\)-order Green's function is 
\[g_{\alpha}(x,y)=\frac{1}{\sqrt{2\alpha}}\left(e^{-\sqrt{2\alpha}|x-y|}-e^{-\sqrt{2\alpha}|x+y|} \right) \leq C x\]
and we have \(R_{2}1(x)=\frac{1}{2}(1-e^{-2x}) \leq 2x .\) Set \(f_n(x):=n^{3/2}{\bf 1}_{(0, 1/n)}\), \(d\mu_n:= f_n dx, d\mu:=0\) and \(A^n, A\) are corresponding PCAFs. Similarly to Example \ref{ex4-2}, we have
\begin{eqnarray*}
\lim_{n\to \infty}\mathbb{E}_m\left[\sup_{t\leq T}|A_t^n-A_t|^2 \right] &\leq &  \lim_{n\to \infty} e^{2T} \int R_21\ f_n\ R_1f_n\ dx\\
&\leq & \lim_{n\to \infty} C n^3 \int_0^{1/n}\int_0^{1/n} x^2 dxdy \\
&=& \lim_{n\to \infty} C \frac{1}{n}\\
&=&0
\end{eqnarray*}
and
\begin{eqnarray*}
\lim_{n\to \infty} \rho(\mu_n, \mu) &=&\lim_{n\to \infty}  \frac{n^3}{\sqrt{2}}\int_0 ^{1/n}\int_0 ^{1/n} (e^{-\sqrt{2}|x-y|}-e^{-\sqrt{2}|x+y|}) dxdy\\
&\geq & \lim_{n\to \infty}  C n^3 \int \int_{\{|x-y|\leq \frac{1}{2n}, |x+y|\geq \frac{1}{n},  0\leq x,y \leq \frac{1}{n}\}} (e^{-\frac{\sqrt{2}}{2n}}-e^{-\frac{\sqrt{2}}{n}}) dxdy\\
&\geq & \lim_{n\to \infty}  C n (1-e^{-\frac{\sqrt{2}}{2n}})\\
&=& C,
\end{eqnarray*}
where \(C\) is a constant. Hence, by Theorem \ref{main2}, \(A^n\) converges to \(A\) in \(L^2(\mathbb{P}_m)\) but neither \(\mu_n\) nor \(A^n\) converge to \(\mu\) in \(\rho\) and \(A\) in \(L^2(\mathbb{P}_{\nu_0})\) respectively.
\end{example}

\begin{example}\label{ex5}
We present the following sufficient condition under which Assumption \ref{ass1} holds.

Any Hunt process for which every point has positive capacity satisfies Assumption \ref{ass1}.
\begin{proof}
Let \(K\) be a compact set, and take a relatively compact open set \(D\) such that \(K \subset D^c\). Then we have \[\mathbb{E}_x[e^{-\zeta}]\leq \mathbb{E}_x[e^{-\sigma_{D^c}}]\] for \(x\in K\), where \(\sigma_{D^c}\) is the first hitting time to \(D^c.\) By \cite[Corollary 3.2.3]{CF}, \(\mathbb{E}_x[e^{-\sigma_{D^c}}]\) is a quasi-continuous version of the 1-equilibrium potential of \(D^c\), and is continuous by the assumption that any polar set is the empty set. Thus we have \(\sup_{x\in K}\mathbb{E}_x[e^{-\sigma_{D^c}}]<1\).
\end{proof}
\end{example}

\begin{appendix}
\section{Proof of Lemma \ref{EnergykappaLem}} \label{sec app}

\begin{proof}
At first, we prove (iii) implies (ii). We assume (iii) and this is equivalent to the convergence of \(\mu_n\) to \(\mu\) in \(\rho\) by Proposition \ref{main2-2}.
In general, \(m \not \in \mathcal{S}_{00}\) because \(m(E)\) may not be finite. However, by using the monotone convergence theorem and \(\|U_1m\|_{\infty}=\|R_11\|_{\infty}\leq 1\), in a similar manner to the proof of Theorem \ref{submain1}, we have
\begin{eqnarray*}
\varlimsup_{n\to \infty}\mathbb{E}_m\left[\sup_{0\leq t\leq T}|M_t^{[U_1\mu_n]}-M_t^{[U_1\mu]}|^2\right] \leq \varlimsup_{n\to \infty} 8 (1+T)\mathcal{E}_1(U_1\mu_n - U_1\mu)=0.
\end{eqnarray*}
Similarly to the proof of Theorem \ref{submain1}, we have

\begin{eqnarray*}
\varlimsup_{n\to \infty}\mathbb{E}_m\left[\sup_{0\leq t\leq T}\left|\int_0^t \widetilde{U_1\mu_n}(X_s)ds-\int_0^t \widetilde{U_1\mu}(X_s)ds\right|^2\right]
&\leq & \varlimsup_{n\to \infty}T e^T \int_E  \left|\widetilde{U_1\mu_n}-\widetilde{U_1\mu}\right|^2 dm\\
&\leq &  \varlimsup_{n\to \infty}T e^T \mathcal{E}_1\left(U_1\mu_n - U_1\mu\right)\\
&=&0
\end{eqnarray*}
and
\begin{eqnarray*}
\varlimsup_{n\to \infty}\mathbb{E}_m\left[\sup_{0\leq t\leq T}\left|\widetilde{U_1\mu_n}(X_0)- \widetilde{U_1\mu}(X_0)\right|^2\right]&=&\varlimsup_{n\to \infty}\int_E \left|\widetilde{U_1\mu_n}(x)-\widetilde{U_1\mu}(x)\right|^2 dm(x)\\
&\leq & \varlimsup_{n\to \infty} \mathcal{E}_1(U_1\mu_n - U_1\mu)\\
&=&0.
\end{eqnarray*}
Moreover, by symmetry of the transition semigroups, we have
\begin{eqnarray*}
\varlimsup_{n\to \infty} \sup_{0\leq t\leq T}\mathbb{E}_m\left[\left|\widetilde{U_1\mu_n}(X_t)- \widetilde{U_1\mu}(X_t)\right|^2\right]
&=&\varlimsup_{n\to \infty}\sup_{t\leq T}
\int_E P_t\left(\left|\widetilde{U_1\mu_n}-\widetilde{U_1\mu}\right|^2 \right)(x)dm(x)\\
&= &\varlimsup_{n\to \infty}\sup_{t\leq T}
\int_E \left|\widetilde{U_1\mu_n}-\widetilde{U_1\mu}\right|^2(x) P_t1(x) dm(x)\\
&\leq &\varlimsup_{n\to \infty} \int_E \left|\widetilde{U_1\mu_n}-\widetilde{U_1\mu}\right|^2(x) dm(x)\\
&=&0.
\end{eqnarray*}
Hence, by the Fukushima decomposition (\ref{eq:NTTU1}), we have
\[\lim_{n\to \infty}\sup_{0\leq t\leq T}\mathbb{E}_m\left[|A_t^n-A_t|^2\right]=0.\] 
Similarly to the above, by (\ref{eq:Ukappa}), we have
\begin{eqnarray*}
\varlimsup_{n\to \infty}\mathbb{E}_{\kappa}\left[\sup_{0\leq t\leq T}|M_t^{[U_1\mu_n]}-M_t^{[U_1\mu]}|^2\right] \leq \varlimsup_{n\to \infty} 8 (1+T)\mathcal{E}_1(U_1\mu_n - U_1\mu)=0.
\end{eqnarray*}
and, by the Beurling-Deny decomposition, we have
\begin{eqnarray*}
\varlimsup_{n\to \infty}\mathbb{E}_{\kappa}\left[\sup_{0\leq t\leq T}\left|\int_0^t \widetilde{U_1\mu_n}(X_s)ds-\int_0^t \widetilde{U_1\mu}(X_s)ds\right|^2\right]
&\leq & \varlimsup_{n\to \infty}T e^T \int_E  \left|\widetilde{U_1\mu_n}-\widetilde{U_1\mu}\right|^2 d\kappa \\
&\leq &  \varlimsup_{n\to \infty}T e^T \mathcal{E}\left(U_1\mu_n - U_1\mu\right)\\
&=&0.
\end{eqnarray*}

By Lemma \ref{symSmooth}, we have 
\begin{eqnarray*}
\varlimsup_{n\to \infty} \sup_{0\leq t\leq T}\mathbb{E}_{\kappa}\left[\left|\widetilde{U_1\mu_n}(X_t)- \widetilde{U_1\mu}(X_t)\right|^2\right]
&= &\varlimsup_{n\to \infty} \sup_{0\leq t\leq T}\mathbb{E}_{\kappa}\left[\left|\mathbb{E}_{X_t}[\widetilde{A_{\infty}^n} -\widetilde{A_{\infty}} ]\right|^2\right]\\
&\leq &\varlimsup_{n\to \infty} \sup_{0\leq t\leq T}\mathbb{E}_{\kappa}\left[\mathbb{E}_{X_t}\left[(\widetilde{A_{\infty}^n} -\widetilde{A_{\infty}} )^2\right]\right]\\
&\leq &\varlimsup_{n\to \infty}
2 \left( \sqrt{\mathcal{E}_1(U_1\mu_n)} +\sqrt{\mathcal{E}_1(U_1\mu)} \right) \sqrt{\mathcal{E}_1(U_1(\mu_n-\mu))}\\
&=&0.
\end{eqnarray*}

Hence, by the Fukushima decomposition (\ref{eq:NTTU1}), we have
\[\lim_{n\to \infty}\sup_{0\leq t\leq T}\mathbb{E}_{\kappa}\left[|A_t^n-A_t|^2\right]=0.\] 
Using the equation \(\int U_2\nu\ d\nu_0 = \int \mathbb{E}_x[e^{-2\zeta} {\bf 1}_{\{\partial\}}(X_{\zeta -})] d\nu(x)\) for \(\nu \in \mathcal{S}_0\) (\cite[Theorem 5.4.3 (iv)]{CF}), similarly to the above, we have
\[\lim_{n\to \infty}\sup_{0\leq t\leq T}\mathbb{E}_{\nu_0}\left[|A_t^n-A_t|^2\right]=0.\]

In the same way as the proof of Lemma \ref{EnergyLem}, (ii) implies (i).

We assume (i). In a similar manner to the proof of Lemma \ref{EnergyLem}, by using Lemma \ref{symSmooth}, we have
\begin{eqnarray*}
\mathbb{E}_{m+\frac{\kappa}{2}+\frac{\nu_0}{2}}\left[|\widetilde{A^n_{\infty}}-\widetilde{A_{\infty}}|^2\right] \leq C_{\varepsilon}\mathbb{E}_{m+\kappa+\nu_0}\left[|\widetilde{A^n_{T}}-\widetilde{A_{T}}|^2\right] + C \varepsilon e^{-2T} \left( \mathcal{E}_1(U_1\mu_n -U_1\mu) + \sqrt{\mathcal{E}_1(U_1\mu_n -U_1\mu)}\right),
\end{eqnarray*}
where \(C_{\varepsilon}=2+\frac{1}{\varepsilon -1}\) for \(\varepsilon >1\). For any \(\delta>0\), we can take \(T>0\) satisfying \(\frac{C}{e^{T}-C} \leq \delta\) and set \(\varepsilon = e^T\). We have
\[\varlimsup_{n\to \infty} \mathbb{E}_{m+\frac{\kappa}{2}+\frac{\nu_0}{2}}\left[|\widetilde{A^n_{\infty}}-\widetilde{A_{\infty}}|^2\right] \leq \delta^2\]
and so (iii) holds.

\end{proof}

\end{appendix}

\section*{Declaration of competing interest}
The author declares that he has no known competing financial interests that could have appeared to influence the work reported in this paper.

\section*{Acknowledgments}
This work was supported by JSPS KAKENHI Grant Number JP25K17270.

\end{document}